\documentclass[english,11pt,a4paper,leqno]{amsart}
\usepackage{amsmath,amstext, amsthm, amssymb}
\usepackage{appendix}

\usepackage{color}
\definecolor{liens}{rgb}{1,0,0}
\usepackage[colorlinks=true, linkcolor=blue, 
hyperfootnotes=true,citecolor=blue,urlcolor=black]{hyperref}

\usepackage{latexsym, amsthm, amsfonts, amsmath,amssymb,graphicx,xcolor}

\usepackage[shortcuts]{extdash} 

\input xy
\xyoption{all}
\usepackage[latin1]{inputenc} 
\usepackage[T1]{fontenc} 
\usepackage[english]{babel} 

\usepackage{pifont}
\usepackage[english]{babel}
\usepackage[all]{xy}
\usepackage{color}

\newtheorem{thm}{Theorem}[section]
\newtheorem{cor}[thm]{Corollary}
\newtheorem{lemma}[thm]{Lemma}
\newtheorem*{claim}{Claim}
\newtheorem{prop}[thm]{Proposition}

\newtheorem{conj}[thm]{Conjecture}

\theoremstyle{definition}
\newtheorem{defn}[thm]{Definition}
\newtheorem{defi}[thm]{Definition}

\theoremstyle{remark}
\newtheorem{rmk}[thm]{Remark}
\newtheorem{rem}[thm]{Remark} 

\newtheorem{ex}[thm]{Example} 

\numberwithin{equation}{section}

\def\beq{\begin{equation}}
\def\eeq{\end{equation}}

\def\crash#1{}

\def\N{{\mathbb N}}
\def\Z{{\mathbb Z}}
\def\Q{{\mathbb Q}}

\def\C{{\mathbb C}}

\def\GL{\mathrm{GL}}

\def\l{\left}
\def\r{\right}

\def\div{\mathrm{div}}

\def\id{\mathrm{id}}

\def\cH{{\mathcal H}}

\def\cL{{\mathcal{ L}}}

\def\s{\sigma}
\def\V{\mathbb{V}}

\def\Hom{\operatorname{Hom}}

\def\G{\mathcal{G}}
\def\Gal{\operatorname{Gal}}
\def\Aut{\operatorname{Aut}}

\def\<{\langle}
\def\>{\rangle}

\def\dim{\mathrm{dim}}
\def\trdeg{\mathrm{tr.deg}}
\def\sphi{\phi\sigma}
\def\sGal{\Gal^{\sigma}}

\def\Quot{\mathrm{Quot}}

\def\sdim{\sigma$-$\dim}
\def\dim{\mathrm{dim}}
\def\strdeg{\s$-$\trdeg}

\newcommand{\f}{\phi}
\newcommand{\fs}{\phi\sigma}

\begin{document}

\sloppy

\title{Algebraic independence and linear difference equations}

\author[B. Adamczewski]{Boris Adamczewski}
\address{Univ Lyon, Universit\'e Claude Bernard Lyon 1, CNRS UMR 5208, Institut Camille Jordan, 43 blvd. du 11 novembre 1918, F-69622 Villeurbanne cedex, France}
\email{boris.adamczewski@math.cnrs.fr}
\author[T. Dreyfus]{Thomas Dreyfus}
\address{Institut de Recherche Math\'ematique Avanc\'ee, U.M.R. 7501 Universit\'e de Strasbourg et C.N.R.S. 7, rue Ren\'e Descartes 67084 Strasbourg, France}
\email{dreyfus@math.unistra.fr}
\author[C. Hardouin]{Charlotte Hardouin}
\address{Institut de Math\'ematiques de Toulouse, Universit\'e Paul Sabatier, 118, route de Narbonne, 31062 Toulouse, France}
\email{hardouin@math.univ-toulouse.fr}
\author[M. Wibmer]{Michael Wibmer}
\address{Institute of Analysis and Number Theory, Graz University of Technology, Kopernikusgasse 24/II,8010 Graz, Austria}
\email{wibmer@math.tugraz.at}

\thanks{ This project has received funding from the European Research Council (ERC) under the European Union's Horizon 2020 research and innovation program under the Grant Agreement No 648132 and from the ANR de rerum natura ANR-19-CE40-0018. The fourth author was supported by the NSF grants DMS-1760212, DMS-1760413, DMS-1760448 and the Lise Meitner grant M-2582-N32 of the Austrian Science Fund FWF. We are grateful for the support from the NSF grant DMS-1952694 that allowed us to work on this project during DART X in New York}

\date{\today}

\maketitle

\begin{abstract}

We consider pairs of automorphisms $(\phi,\s)$ acting on fields of Laurent or Puiseux series:  
pairs of shift operators $(\phi\colon x\mapsto x+h_1,\ \linebreak[4] \s\colon x\mapsto x+h_2)$, of $q$-difference operators 
$(\phi\colon x\mapsto q_1x,\ \s\colon x\mapsto q_2x)$, 
and of Mahler operators $(\phi\colon x\mapsto x^{p_1},\ \s\colon x\mapsto x^{p_2})$. 
Given a solution $f$ to a linear $\phi$-equation and a solution $g$ to a linear 
$\sigma$-equation, both transcendental, 
we show that 
$f$ and $g$ are algebraically independent over the field 
of rational functions, assuming that the corresponding parameters are sufficiently independent.   
As a consequence, we settle a conjecture about Mahler functions put forward 
by Loxton and van der Poorten in 1987. We also give an application  
 to the algebraic independence of  $q$\=/hypergeometric functions.  
 Our approach provides a general strategy to study this kind of question and 
 is based on a suitable Galois theory: the $\s$-Galois theory of linear $\phi$-equations.  
\end{abstract}
	

\section{Introduction}	
Let $K$ be a field and let $F$ be a field extension endowed with an endomorphism 
$\phi$ such that $\phi(K)\subset K$. 
A linear $\phi$-difference equation over $K$ is an equation of the form 
\begin{equation}\label{eq: phi0}
\phi^{n} (y) +a_{n-1}\phi^{n-1}(y)+ \cdots + a_{0} y=0 \,,
\end{equation}
where  $a_{0},...,a_{n-1} \in K$. 
The set ${\rm Sol}_{\phi,K,F}$, formed by all elements in $F$ that are solution to a linear $\phi$-difference 
equation over $K$, is a ring  containing $K$.  Traditionally, the algebraic relations over $K$ between the elements of 
${\rm Sol}_{\phi,K,F}$ are studied through the difference Galois theory associated 
with the endomorphism $\phi$ 
(see  for instance, \cite{VdPdifference}).  
If we assume that $K$ and $F$ are endowed with a second endomorphism $\s$ 
that is \emph{sufficiently independent} from $\phi$, we expect the intersection of 
${\rm Sol}_{\phi,K,F}$ and ${\rm Sol}_{\s,K,F}$ to be \emph{small},  
even possibly reduced to $K$.  
Sch\"afke and Singer \cite{SchaefkeSinger} recently confirmed this expectation in several important cases. (See Theorem \ref{lem:linearlyclosedfieldextension} below for a precise statement.)  They consider pairs of shift operators $(\phi\colon x\mapsto x+h_1,\ \s\colon x\mapsto x+h_2)$, of $q$-difference operators 
$(\phi\colon x\mapsto q_1x,\ \s\colon x\mapsto q_2x)$, 
and of Mahler operators $(\phi\colon x\mapsto x^{p_1},\ \s\colon x\mapsto x^{p_2})$.
While some special cases were already known, these authors were the  
first to provide a unified approach to this type of results.  

The aim of this paper is to go one step further, promoting the idea that, in the above cases,  
the elements of ${\rm Sol}_{\phi,K,F}$ and ${\rm Sol}_{\s,K,F}$  should be algebraically independent, 
unless they belong to the small intersection. 
Our main result, Theorem \ref{thm: main}, appears to be the first 
general result supporting this viewpoint. Our approach provides a unified strategy for attacking this kind of 
problem. It rests on a suitable Galois theory: the $\s$-Galois theory of linear $\phi$-difference equations
developed in \cite{OvWib}. 
As a particular instance of  Theorem \ref{thm: main},  
we settle a conjecture of Loxton and van der Poorten \cite{vdP87} concerning Mahler functions.  
The latter, which was itself motivated by its consequence in the theory of finite automata, 
was our initial motivation for the present work.  
We also give a second application of Theorem \ref{thm: main} to the algebraic independence of 
$q$-hypergeometric series. The strategy we follow to prove Theorem~\ref{thm: main} was initiated recently by the first three authors in \cite{ADH}, where general hypertranscendence results are obtained 
for solutions of linear difference equations associated with 
the same three operators (shift, $q$-difference, and Mahler).  However, 
when working with a parametric operator (here $\s$) that is an endomorphism, 
instead of a derivation as in \cite{ADH}, one needs to overcome 
a number of technical difficulties.  
Let us also mention that we introduce in Section \ref{sec: main} some new group theoretic arguments concerning linear algebraic groups that would lead to a significant simplification 
of the proof of the main result of \cite{ADH}.  

\subsection{Statement of our main result} 

Throughout this paper, our framework consists of a tower of field extensions 
$\C\subset K \subset F$ 
with the following properties.  
\begin{itemize}
\item The field $K$ is equipped with a pair of automorphisms $(\phi,\s)$.  
\item These automorphisms extend to  $F$.
\end{itemize}
Specifically, we consider the  following situations, which we refer to as Cases 
\textbf{2S},  \textbf{2Q}, and \textbf{2M}, respectively.

\medskip

\noindent {\bf Case} \textbf{2S}. In this case, we consider $K=\C(x)$, $F=\C((x^{-1}))$, 
$\phi(x)=x+h_1$ and $\s(x)=x+h_2$, where $h_1,h_2\in \C$ are \emph{$\Z$-linearly independent}, i.e. 
$h_1/h_2\not\in \mathbb Q$.

\medskip

\noindent  {\bf Case} $\textbf{2Q}$.  In this case, we let $K=\displaystyle\bigcup_{j\geq 1} \C(x^{1/j})$ 
denote the field of ramified rational functions. We also use the notation $\C(x^{1/*})$ for this field. 
We let $F=\displaystyle\bigcup_{j\geq 1} \C((x^{1/j}))$ 
denote the field of Puiseux series. We also use the notation $\C((x^{1/*}))$ for this field. 
We let $(\phi,\s)$ denote the pair of automorphisms of $K$ (and $F$) defined by 
$$
{\phi(x)= q_1x} \quad \mbox{and} \quad \s(x)= q_2x\,,
$$
where $q_1$ and $q_2$ are two \emph{multiplicatively independent} nonzero complex numbers, i.e. $q_1^{n_1}q_2^{n_2}=1$ implies $n_1=n_2=0$ for all $n_1,n_2\in\Z$. Furthermore, we also add the following mild restriction: 
$q_1$ and $q_2$ cannot be both algebraic numbers of modulus one, whose Galois conjugates all have modulus one.	
For instance, one cannot choose $q_1=(3+4i)/5$ and $q_2=(5+12i)/13$.  Note that, when $q_1$ and $q_2$ are multiplicatively independent, none of them is a root of unity.

\medskip

\noindent  {\bf Case} $\textbf{2M}$.  In this case, we let $K=\C(x^{1/*})$, $F=\C((x^{1/*}))$, and  
we let $(\phi,\s)$ denote the pair of automorphisms of $K$ (and $F$) defined by 
$$
\phi(x)= x^{p_1} \quad \mbox{and} \quad \s(x)= x^{p_2}\,,
$$
where $p_1$ and $p_2$ are two \emph{multiplicatively independent} natural numbers.

The recent result of Sch\"afke and Singer \cite{SchaefkeSinger} mentioned before 
can now be stated as follows. 

\begin{thm}[Sch\"afke and Singer]\label{lem:linearlyclosedfieldextension}
Let $K$, $F$, and $(\phi,\s)$ be defined as in Cases \textbf{2S}, 
\textbf{2Q}, and 
\textbf{2M}.  Then an element $f\in F$ cannot satisfy both a linear $\phi$-difference equation and 
a linear $\s$-difference equation with coefficients in $K$, unless it belongs to $K$.
\end{thm}

\begin{rem}\label{rem: SS}
Using the same reasoning as in \cite[Remark 1.4]{ADH}, Theorem~\ref{lem:linearlyclosedfieldextension} may be deduced from \cite[Corollary 14 to 16]{SchaefkeSinger}. 
We recall that Case $\textbf{2Q}$  was proved, though in less generality, by B\'ezivin and Boutabaa \cite{BeBo}, while Case $\textbf{2M}$  is due to the first author and Bell \cite{BorisAboutMahler}.  We refer the reader to \cite{SchaefkeSinger} and the references therein for more details 
on the different contributions to the other cases. 
\end{rem}

Our main result is the following generalization of Theorem \ref{lem:linearlyclosedfieldextension}.

\begin{thm}\label{thm: main}
Let $K$, $F$, and $(\phi,\s)$ be defined as in Cases \textbf{2S}, \textbf{2Q}, and \textbf{2M}. Let $f\in F$ be a solution to a linear $\phi$-difference equation with coefficients in $K$ and let $g\in F$ be a solution to a linear $\s$-difference equation with coefficients in $K$. Then $f$ and $g$ are algebraically independent over $K$, unless one of them belongs to $K$.
\end{thm}

Though expressed differently\footnote{See Theorem \ref{thm: sigmatranscendence} 
for a statement close to Theorem \ref{thm: main} but involving $\s$-transcendence as in \cite{DHRqdiffhypergo}.}, 
Case \textbf{2Q} of Theorem \ref{thm: main} generalizes some results of \cite{DHRqdiffhypergo}, where  additional assumptions are made on the Galois group 
associated with the linear $\phi$-difference equation satisfied by $f$.

\begin{rem}
In the three cases we consider, the independence required for the corresponding pairs $(\phi,\s)$ could be rephrased and unified as follows: $\phi$ and $\s$ commute, and if $f\in F$ is not a constant (i.e., $f\not\in \mathbb C$), then the equation $\phi^n(f)=\s^m(f)$ 
has no nontrivial solution in integers $n$ and $m$\footnote{The additional mild condition in Case $\textbf{2Q}$ is 
probably not necessary.}. As suggested by the previous discussion, it would be interesting to formalize a general principle emerging from Theorems \ref{lem:linearlyclosedfieldextension} and \ref{thm: main}. 
\end{rem}

\subsection{Application to Mahler functions}

Let $\mathbb K$ be a field and $p\geq 2$ be an integer. 
A power series $f(x)\in \mathbb K[[x]]$ is a $p$-\emph{Mahler function}
if there exist  polynomials $a_0(x),\ldots,a_n(x)\in \mathbb K[x]$, not all zero, such that
\begin{equation}\label{eq: mahler}
a_0(x)f(x)+a_1(x)f(x^p)+\cdots+a_n(x)f(x^{p^n})=0\,.
\end{equation} 
Furthermore, $f$ is a \emph{Mahler function} if it is a $p$-Mahler function for some 
integer $p\geq 2$.  
Regarding Equation \eqref{eq: mahler}, it is tempting to ask about the significance of 
the parameter $p$. In 1976, van der Poorten \cite{vdP76} suggested that Mahler functions associated 
with multiplicatively independent parameters should behave independently. 
For instance, he asked whether the two Mahler functions
$$
\sum_{n=0}^{\infty} x^{2^n}  \quad\mbox{ and }\quad\sum_{n=0}^{\infty} x^{3^n}
$$
are algebraically independent over $\overline{\mathbb Q}(x)$.  
The transcendence theory of Mahler functions, that is, the study of the transcendence and algebraic 
independence of values of Mahler functions over $\overline{\mathbb{Q}}$ at algebraic points, was initiated by Mahler at the end of the 1920s. 
Mahler only considered (possibly inhomogeneous) order one equations, and Mahler's method 
really restarted in the 1970s in the hands of Kubota, 
Loxton and van der Poorten, and later Ku. Nishioka (see the recent survey \cite{Ad19} and the references therein).  
Kubota, Loxton, and van der Poorten expected that Mahler's method   
could be developed far enough to solve the above question\footnote{In \cite{Ku76}, Kubota announced a paper on this problem but the latter never appeared. Also, Loxton and van der Poorten stated some related results in 
\cite{LvdP78}, but the corresponding proofs are considered to be incomplete (see the discussion in \cite[p.  89]{Ni94}).}.   
In 1987, van der Poorten \cite{vdP87} announced a series of results that he expected to prove 
in his collaboration with Loxton on Mahler's method\footnote{However, after this date, there is no 
new publication on this topic by these authors.}.   
Among them, the following two are related to Mahler functions associated with multiplicatively independent parameters.  

\begin{conj}[Loxton and van der Poorten]\label{conj}
Let $p_1$ and $p_2$ be two multiplicatively independent natural numbers.  
Then the following holds. 
\begin{itemize}
\item[{\rm (i)}] An element $f\in \overline{\mathbb Q}[[x]]$ cannot be both a $p_1$- and a $p_2$-Mahler function,  unless it is rational. 

\item[{\rm (ii)}] If $f\in \overline{\mathbb Q}[[x]]$ is a $p_1$-Mahler function and 
$g\in \overline{\mathbb Q}[[x]]$ is a $p_2$-Mahler function. Then  
 $f$ and $g$ are algebraically independent over $\overline{\mathbb Q}(x)$, unless one of them is rational. 
\end{itemize}
\end{conj}

We find that (ii) is a generalization of (i), while (i) itself is a generalization of Cobham's theorem, 
an important result in the theory of automatic sequences and sets 
(see, for instance, \cite{BorisAboutMahler} for more details).  
As already mentioned in Remark \ref{rem: SS}, 
part (i) of Conjecture~\ref{conj} has finally been proved by the first author and Bell in 
\cite{BorisAboutMahler}. A different proof has been given by Sch\"afke and Singer 
in \cite{SchaefkeSinger}. 

The first result towards Part (ii) of Conjecture \ref{conj} is due to Ku. Niskioka \cite{Ni94} 
who proved the case where $f$ and 
$g$ both satisfy an inhomogeneous order one equation.  
Recently, the first author and Faverjon \cite{AF18a,AF18b} proved the case where $f$ and $g$ are both solution to a Mahler equation of the form \eqref{eq: mahler} with $a_0(0)a_n(0)\not=0$.  The results of \cite{Ni94}  and \cite{AF18a,AF18b} are based on Mahler's method. 
As a direct consequence of Theorem \ref{thm: main}, we fully prove Conjecture \ref{conj} 
(even a more general version where the base field is the field of complex numbers). 

\begin{thm}
Let $p_1$ and $p_2$ be two 
multiplicatively independent natural numbers. 
Let $f\in {\mathbb C}[[x]]$ be a $p_1$-Mahler function and let 
$g\in {\mathbb C}[[x]]$ be a $p_2$-Mahler function. Then  
$f$ and $g$ are algebraically independent over ${\mathbb C}(x)$, unless one of them is rational.  
\end{thm}

Let us also mention that, in an ongoing work, Medvedev, Nguyen, and Scanlon \cite{MNS20+} 
study a similar problem
where linear Mahler equations of arbitrary order are replaced by
nonlinear order one equations of the form $f(x^p)=P(x,f(x))$, where $P$
is a polynomial in two variables. 
\subsection{Application to $q$-hypergeometric series} 

Let $q$ be a nonzero complex number that is not a root of unity.  
For every nonnegative integer $n$, we let 
$$
[n]!_q=\prod_{i=1}^n\frac{1-q^i}{1-q}
$$
denote the $q$-analog of the factorial $n!$.  
The $q$-Pochammer symbol, also called $q$-shifted factorial, is defined as
$$
(a;q)_n:=\prod_{i=0}^{n-1}(1-aq^i) \,.
$$
The most classical $q$-analog of hypergeometric functions is given by the so-called (generalized) \emph{basic hypergeometric functions} or $q$-\emph{hypergeometric functions}:
$$
{}_{r+1}\phi_s\left(\begin{array}{c}\alpha_1,\dots,\alpha_r\\ \beta_1,\dots,\beta_s\end{array}\Bigg|\,q,x\right)=\sum_{n=0}^\infty\frac{(\alpha_1;q)_n\cdots(\alpha_r;q)_n}{(\beta_1;q)_n\cdots(\beta_s;q)_n}\left((-1)^nq^{\binom{n}{2}}\right)^{s-r}\frac{x^n}{[n]!_q}\, \cdot
$$
Here, $r$ and $s$ are two nonnegative natural numbers and $\alpha_1,\ldots,\alpha_r,\beta_1,\ldots,\beta_s$ 
are complex numbers. For the coefficients to exist we must assume that $\beta_i \not\in q^{\mathbb Z_{\leq0}}$ 
for all $i$. 
By definition, a $q$-hypergeometric function belongs to $\mathbb C[[x]]$. 
It is also well-known that a $q$-hypergeometric function satisfies a $q$-linear difference equation 
(see, for instance, \cite{GR04}).   
As a direct consequence of Case $\textbf{2Q}$ of Theorem \ref{thm: main}, we obtain the following result. 

\begin{thm}
Let $q_1$ and $q_2$ be two multiplicatively independent complex numbers satisfying 
the conditions given in Case $\textbf{2Q}$.  
Let $f$ be a $q_1$-hypergeometric function and  let $g$ be a $q_2$-hypergeometric function. Then 
$f$ and $g$ are algebraically independent over 
$\mathbb C(x)$, unless one of them is rational.  
\end{thm}

\subsection{Organization of the paper.} 
This article is organized as follows. In
Section \ref{sec: galois}, we give a short introduction to the Galois theory and
the $\s$-Galois theory of linear $\phi$-difference equations, following \cite{VdPdifference} and \cite{OvWib}. 
Several auxiliary results are gathered in Section \ref{sec: aux}. Section \ref{sec: main} is devoted to the proof of Theorem \ref{thm: main}. 
The latter is obtained as a consequence of Theorem~\ref{thm: sigmatranscendence}, which generalizes 
some results of \cite{DHRqdiffhypergo}. 
Finally, Appendix~\ref{appendix}  provides the reader with some complementary results on difference algebraic 
groups and their representations that are needed for the proof of Theorem~\ref{thm: main}.
\section{Galois theories of linear difference equations}\label{sec: galois}

In this section, we provide a short introduction to the Galois theory of linear $\phi$-equations and to 
the $\s$-Galois theory of linear $\phi$-equations.

\subsection{Galois theory of linear difference equations}\label{sec:PVring}
We first recall some notation, as well as classical results, concerning the Galois theory of linear difference equations. We refer the reader to \cite{VdPdifference} for more details.

\subsubsection{Notation in operator algebra}\label{sec: operator}
In what follows, all rings are commutative, with identity, and contain $\Q$.  
In particular, all fields are of characteristic zero. 
Given a ring $R$, we let  $\Quot(R)$ denote the total quotient ring of $R$, that is, the localization of $R$ 
at the multiplicatively closed subset of all nonzero divisors. Given a field $K$,  a $K$-algebra $R$,  
and a subset $S$ of $R$, we let $K(S)$ denote the total quotient ring of the ring generated  by $S$ over $K$.  
 
 A $\phi$-difference ring, or $\phi$-ring for short, is a pair $(R,\phi)$ where $R$ is a ring and 
 $\phi$ is a ring endomorphism of $R$.  When $R$ is a field,  $R$ is called a $\phi$-field.  
 An ideal $I$ of $R$ such that $\phi(I)\subset I$ is called a difference ideal or a $\phi$-ideal. 
The difference ring $(R,\phi)$ is simple if the only $\phi$-ideals of $R$ are $\{0\}$ and $R$. 
Two difference rings $(R_1,\phi_1)$ and $(R_2,\phi_2)$ are isomorphic 
if there exists a ring isomorphism $\varphi$ 
between $R_1$ and $R_2$ such that $\varphi\circ \phi_1=\phi_2\circ \varphi$.  
A difference ring $(S,\phi')$ is a difference ring extension of $(R,\phi)$ if $S$ is a ring extension of $R$ 
and if $\phi'_{\mid R}=\phi$. In this case, we usually keep on denoting $\phi'$ by $\phi$.  
When $R$ is a $\phi$-field, we say that $S$ is a $R$-$\phi$-algebra. 
Two difference ring extensions $(R_1,\phi)$ and $(R_2,\phi)$ of the difference ring $(R,\phi)$ 
are isomorphic over $(R,\phi)$ if there exists a difference ring isomorphism $\varphi$ 
from $(R_1,\phi)$ to $(R_2,\phi)$ such that $\varphi_{\mid R}=\mbox{Id}_{R}$.  
The ring of constants of the difference ring $(R,\phi)$ is defined by 
$$
R^{\phi}:=\{r\in R \mid \phi(r)=r\} \,.
$$
If $R^{\phi}$ is a field, it is called the field of constants. 
If there is no risk of confusion, we usually simply say that $R$, instead of $(R,\phi)$, 
is a difference ring (or a difference field, or a difference ring extension...).

\subsubsection{Difference equations and linear difference systems}  
A linear $\phi$-equation  of order $n$ over a $\f$-field $K$ is an equation of the form 
\begin{equation}\label{eq: phi}
 \cL(y):=\phi^{n} (y) +a_{n-1}\phi^{n-1}(y)+ \cdots + a_{0} y=0 \,,
\end{equation}
with  $a_{0},...,a_{n-1} \in K$. If   $a_0 \neq 0$, 
this relation  can be written in matrix form as 
\begin{equation}\label{eq: companion}
\phi(Y)=A_\cL Y\,
\end{equation}
where
$$
 A_\cL :=\begin{pmatrix}
0&1&0&\cdots&0\\
0&0&1&\ddots&\vdots\\
\vdots&\vdots&\ddots&\ddots&0\\
0&0&\cdots&0&1\\
-a_{0}& -a_{1}&\cdots & \cdots & -a_{n-1}
\end{pmatrix} \in \GL_{n}(K) \,.
$$
The matrix $A_\cL$ is called the companion matrix associated with Equation~\eqref{eq: phi}. 
It is often more convenient to use the notion of linear difference system,  
that is of system of the form 
\begin{equation}\label{eq:systeminitial}
\phi(Y)=AY, \hbox{ with } A \in \GL_n(K).\,
\end{equation}
We recall that two difference systems $\phi (Y)=AY$ and $\phi (Y)=\widetilde{A}Y$ with $A,\widetilde{A} \in \GL_{n}(K)$ are 
said to be \emph{equivalent over $K$} if  there exists a gauge transformation $T \in \GL_{n}(K)$ such that 
$\widetilde{A}=\phi(T) AT^{-1}$. In that case, ${\phi (Y)=AY}$ if and only if $\phi (TY)=\widetilde{A} (TY)$. 
 
\begin{rem}\label{rem:orderofthesystorderoftheequation}
Let $L|K$ be an extension of $\f$-fields and assume that $f=(f_1,f_2,\dots,f_{n})^\top \in L^n$ is a  
solution to $\phi(Y)=AY$, with $A \in \GL_n(K)$.
As the $K$-subspace of $L$ generated by $f_1,\ldots,f_n$ is closed under $\f$, it follows that each coordinate 
$f_i$ of $f$ satisfies a nontrivial linear $\phi$-equation over $K$ of order at most $n$. 
\end{rem}

\subsubsection{Galois theory of linear difference equations}

 In this section, we give a brief summary of the Galois theory of linear 
 difference equations.

  \begin{defn}[Definition 2.2 in \cite{OvWib}]\label{def:pseudofield}
 A $\phi$-pseudo field is a $\phi$-simple, Noetherian $\phi$-ring $K$ such that every nonzero divisor of 
 $K$ is invertible in $K$. If $K$ is a $\phi$-pseudo field then there exists 
 orthogonal idempotents $e_1,\dots, e_d$ of $K$  such that 
 \begin{itemize}
 \item $K=e_1.K \oplus \hdots \oplus e_d.K$,
 \item $\phi(e_1)=e_2, \phi(e_2)=e_3, \dots, \phi(e_d)=e_1$ and 
 \item $e_i.K$ is a field for $i=1,\dots,d$ (so, $e_i.K$ is a $\phi^d$-field).
 \end{itemize}
 
 \end{defn}

Let $K$ be a $\phi$-field,  $R$ be a $K$-$\phi$-algebra and $S \subset R$ be a subset of $R$. 
We let $K\{S\}_\phi$ denote the smallest $K$-$\phi$-subalgebra of $R$ that contains $S$. 
If $R$ is a $\phi$-pseudofield, we let $K\langle S \rangle_\phi$ denote 
the smallest $\phi$-pseudo field of $R$ that contains $S$.

A \emph{Picard-Vessiot ring} for \eqref{eq:systeminitial} over  a difference field  $(K, \phi)$ 
is a $K$-$\phi$-algebra $R_A$ satisfying the following three properties. 

\smallskip
\begin{itemize}
\item[(1)] There exists $U \in \GL_{n}(R_A)$ such that $\phi(U)=AU$. Such a matrix $U$ is called a
\emph{fundamental matrix}.  

\smallskip

\item[(2)] $R_A$ is generated as a $K$-algebra by the coordinates of $U$ and by $\det(U)^{-1}$,  that is 
$R_A=K[U,\det(U)^{-1}]$. 

\smallskip

\item[(3)] $R_A$ is a simple $\f$-ring.

\end{itemize}

\medskip

A \emph{Picard-Vessiot extension} $K_A$ for \eqref{eq:systeminitial} over the difference field $(K,\phi)$ 
is a $\phi$-pseudo field extension of $K$ satisfying the following properties. 

\smallskip

\begin{itemize}
\item $K_A=K( U )$ where $U \in \GL_n(K_A)$ is a fundamental matrix for \eqref{eq:systeminitial}.

\smallskip

\item $K_A^{\phi}=K^\phi$.
\end{itemize}

\medskip

Note that we have $K^\f=F^\f=\C$ in our main cases of interest, i.e. in Cases \textbf{2S}, \textbf{2Q}, and \textbf{2M}. By \cite[Section 1.1]{VdPdifference} and
\cite[Proposition~2.14, Corollary~2.15]{OvWib}, we obtain the following proposition that connects the two definitions. 
\begin{prop}\label{prop:PVRingPVextconstantalgclos}
If $C=K^\phi$ is an algebraically closed field, then there exists a unique  
(up to isomorphism of $K$-$\phi$-algebras)  
Picard-Vessiot extension for \eqref{eq:systeminitial}. 
Moreover,  the following properties hold. 
\begin{itemize}
\item Given a Picard-Vessiot extension $K_A$ with fundemental matrix $U \in \GL_n(K_A)$, the $K$-$\phi$-algebra $R_A:=K[U, \frac{1}{\det U} ] \subset K_A$ is a Picard-Vessiot ring.
\item Given a Picard-Vessiot ring $R_A$, the total quotient ring $\Quot(R_A)$ is a Picard-Vessiot extension.
\end{itemize}
\end{prop}

From now on, we assume that $K$ is a $\phi$-field with $C=K^\phi$ an algebraically closed field of 
characteristic zero. Let $A \in \GL_n(K)$ and let $K_A$ be a Picard-Vessiot extension for $\phi(Y)=AY$. 
Let $R_A \subset K_A$ denote the Picard-Vessiot ring defined in Proposition \ref{prop:PVRingPVextconstantalgclos}. The \emph{Galois group} $\Gal(K_A|K)$ of $\f(Y)=AY$ or of $K_A|K$ is the functor
from the category of $C$-algebras to the category of groups that associates to any $C$-algebra $B$ the group of all $K\otimes_C B$-$\f$-automorphisms of $R_A\otimes_C B$. 
Here $\f$ acts as the identity on $B$. The functor $\G=\Gal(K_A|K)$ is representable. In fact, $\G$ is represented by $C[\G]=(R_A\otimes_K R_A)^\f$ (see \cite[Theorem~2.8]{Bachmaier:Nori} or \cite[Section~1.2]{VdPdifference}).

Let  $U \in \GL_n(R_A)$ be  
a fundamental matrix for 
$\phi(Y)=AY$ and let $B$ be a $C$-algebra. For any $\tau \in \G(B)$, there exists  
$[\tau]_U \in \GL_n(B)$ such that 
$\tau(U\otimes 1)=(U\otimes 1)[\tau]_U$. The morphism of functors $\G\to \GL_{n,C}$, given by $\G(B)\to \GL_n(B),\ \tau\mapsto [\tau]_B$ for any $C$-algebra $B$,
identifies $\G$ with a closed subgroup of $\GL_{n,C}$. 
Since $C$ is algebraically closed and of characteristic zero, an algebraic group $\G$ over $C$ can be identified with $\G(C)$. Therefore one often identifies $\Gal(K_A|K)$ with $\Gal(K_A|K)(C)$.

In that setting, there is a Galois correspondence. We will only need the following special case.

\begin{prop}[Lemma 1.28 in \cite{VdPdifference}] \label{prop:Galois correspondence}
With the above notation, let $\mathcal{H} \subset \Gal(K_A|K)$ be a closed subgroup, then  
the following two statement are equivalent. 

\smallskip

\begin{itemize}
\item $K_A^\mathcal{H} :=\{f \in K_A |\ \tau(f)=f \,,  \;\; \forall\  \tau \in\mathcal{H} \}= K$.

\smallskip

\item $\mathcal{H}=\Gal(K_A|K)$.
\end{itemize}
\end{prop}

We will also need the following result. 

\begin{prop}[{\cite[Theorem 3]{FrankePicardVessiottheory}}, {\cite[Theorem 12.6]{AmanoMasuokaTakeuchi:HopfPVtheory}}]
\label{cor:Galoisgroupconnected}
If moreover $K_A$ is a field, then:

\smallskip

\begin{itemize}
\item $\Gal(K_A|K)$ is connected if and only if $K$ is relatively algebraically closed in $K_A$.

\smallskip

\item If $\mathcal{N}$ is a normal closed subgroup of $\Gal(K_A|K)$, then $K_A^{\mathcal{N}}$ is a 
Picard-Vessiot extension 
for some system $\phi(Y)=A'Y$, where $A' \in \GL_{n'}(K)$ and $n'\leq n$ is a positive integer.
\end{itemize}
\end{prop}

\subsection{$\sigma$-Galois theory of linear difference equations }
\label{subsec:paramdiffgaldiscrete}

In this section, we consider rings endowed with two endomorphisms that commute. 
More formally, a  $\sphi$-ring is a triple $(R,\phi,\s)$ where $R$ is a ring and $(\phi,\s)$ 
is a pair of endomorphisms of $R$ 
that commutes. The notions of Section \ref{sec: operator} extend to $\sphi$-rings in a straightforward fashion.  
Given a $\sphi$-field $K$ and $A \in \GL_n(K)$, the $\s$-Galois theory developed in \cite{OvWib}
aims to understand the algebraic relations between  the solutions of $\phi(Y)=AY$ and their successive transforms 
with respect to $\s$ from a Galoisian point of view. In this section,  we assume that $K$ is a $\sphi$-field of characteristic zero such that $C=K^\phi$ is algebraically closed. Note that $C$ is a $\s$-field because $\s$ and $\f$ commute.

 A $\phi$-pseudo $\sigma$-field $L$ is a $\sphi$-ring that is a $\phi$-pseudo field.  
 The following definition is concerned with the notion of \emph{minimal ring of solutions} 
 in the context of parametrized difference equations. 
 It summarizes in our context \cite[Definition 2.18 and Proposition 2.21]{OvWib}.

\begin{defi} \label{defi:sPVring}
Let  $A \in \GL_n(K)$. A $\phi$-pseudo $\sigma$-field  extension $L_A$ of $K$ is a 
$\sigma$-Picard-Vessiot extension for $\phi(Y)=AY$ over $K$  if 
there exists $U \in \GL_n(L_A)$ such that $\phi(U)=AU$, $L_A=K\langle U\rangle_{\s}$, 
and $L_A^{\phi}=K^\phi$. The $K$-$\sphi$-algebra $S_A= K\{U, \frac{1}{\det(U)}  \}_{\s}$ is called a 
$\s$-Picard-Vessiot ring for   $\phi(Y)=AY$ (over $K$). 
The $\phi$-ring $S_A$ is  $\phi$-simple and $L_A$ is the total quotient ring  of $S_A$.
\end{defi}

 The following definition introduces the $\s$-Galois group.

 \begin{defi}[\cite{OvWib}, Definition 2.50]\label{def:spgaloisgroup}
 Let $A \in \GL_n(K)$ and let ${L_A=K\langle U\rangle_{\s}}$ be a $\s$-Picard-Vessiot extension for 
 $\phi(Y)=AY$. Set $S_A=K\{U,\frac{1}{\det(U)}\}_{\s}$. The $\s$-Galois $\sGal(L_A|K)$ of $L_A$ over $K$ is the functor from the category of $C$-$\s$-algebras to the category of groups that associates to any $C$-$\s$-algebra $B$ the group of all $K\otimes_C B$-$\f\s$-automorphism of $S_A\otimes_C B$.
Here $\phi$ acts as the identity on $B$.
 \end{defi}

 It is proved in \cite[Lemma 2.51]{OvWib} that this functor is represented by a finitely $\s$-generated $C$-$\s$-algebra.
 Therefore, $\sGal(L_A|K)$ is a $\s$-algebraic group over $C$ in the sense of Definition \ref{defi: salgebraic group}. For a brief introduction to $\s$-algebraic groups we refer to Section \ref{sec:sigmagroupscheme} of the appendix. The $C$-$\s$-algebra $C\{G\}$ representing $G=\sGal(L_A|K)$ can explicitly be described as $C\{G\}=(S_A\otimes_K S_A)^\f$.
 Moreover  by \cite[Lemma 2.41]{OvWib},
 \begin{equation} \label{eq:torsor}
 S_A\otimes_K S_A=S_A\otimes_C C\{G\}.
 \end{equation}
If $U\in\GL_n(S_A)$ is a fundamental  matrix for $\f(Y)=AY$, then $C\{G\}=(S_A\otimes_K S_A)^\f=C\{Z,\frac{1}{\det(Z)}\}_\s$ where $Z=(U\otimes 1)^{-1}(1\otimes U)\in\GL_n(S_A\otimes_K S_A)$.

For a $C$-$\s$-algebra $B$, the action of $\tau\in\sGal(L_A|K)(B)$ on $S_A\otimes_C B$ is determined by $\tau(U\otimes 1)=(U\otimes 1)[\tau]_U$, where $[\tau]_U$ denotes the image of $Z$ in $\GL_n(B)$ under the morphism $C\{Z,\frac{1}{\det(Z)}\}_\s\to B$ corresponding to $\tau$ under $G(B)\simeq \Hom(C\{Z,\frac{1}{\det(Z)}\}_\s,B)$. The assignment $G(B)\to \GL_n(B),\ \tau\mapsto [\tau]_U$ identifies $G$ with a $\s$-closed subgroup of $\GL_{n,C}$, in the sense of Definition~\ref{defi: sclosed subgroup}.

In this setting, there is a complete Galois correspondence. However, we will only need the following special case of \cite[Theorem 2.52]{OvWib}.
 
\begin{prop}\label{propo:spgaloiscorresptransdeg}
	With notation as in Definitions \ref{defi:sPVring} and \ref{def:spgaloisgroup}, we have
	$$\{s\in S_A|\ \tau(s\otimes 1)=s\otimes 1, \ \forall\  \tau\in \sGal(L_A|K)(B), \ \forall \text{ $C$-$\s$-algebras } B \}=K.$$
\end{prop}  
 
 The following definition introduces the difference analog of the transcendence degree of field extensions, see
 \cite[ Definition 4.1.7]{Levin}.

\begin{defn}\label{defn:sigmaalgandsigmatransdegree}
Let $L|K$ be a $\sigma$-field extension. Then $a_1,\dots, a_d \in L$ are  
\emph{$\sigma$-algebraically independent} over $K$  if the elements $\sigma^{i}(a_{j})$ ($i\in \N,\ 1\leq j \leq d)$ 
are algebraically independent over $K$. A single element $a\in L$ is also called \emph{$\s$-transcendental} over $K$ if it is $\sigma$-algebraically independent over $K$. Otherwise, $a$ is called \emph{$\s$-algebraic} over $K$.\par 
A $\sigma$-transcendence  basis  of $L$ over $ K$ is a 
maximal  subset of $L$ formed by $\sigma$-algebraically  independent elements over $K$. 
Any two $\sigma$-transcendence bases of $L|K$ have the same cardinality (\cite[Proposition 4.1.6]{Levin}) and so we can define the 
\emph{$\sigma$-transcendence degree} $\strdeg(L|K)$ of $L|K$ as the cardinality of any 
$\sigma$-transcendence basis of $L$ over $K$. 
\end{defn}

The following lemma will be used several times in the sequel.

\begin{lemma}\label{lem2}
Let $L|K$ be a $\sigma$-field extension and let $s\geq 1$ be an integer. Then $f\in L$ is $\s$-algebraic over $K$ if and only if $f$ is $\s^{s}$-algebraic over $K$.
\end{lemma}

\begin{proof}
 If $f$ is $\s$-algebraic over $K$ then the transcendence degree of $K\langle f \rangle_{\s}$ over $K$ is finite. Since $K \subset K\langle f \rangle_{\s^s} \subset K\langle f \rangle_{\s}$, the transcendence degree of $K\langle f \rangle_{\s^s}$ over $K$ is finite and $f$ is $\s^s$-algebraic over $K$.  The converse is obviously true, if $f$ is $\s^s$-algebraic over $K$, it is $\s$-algebraic over $K$.\end{proof}

The $\s$-dimension $\sdim(G)$ of a $\s$-algebraic group is defined in Definition~\ref{defi:sdim}.

\begin{prop}\label{prop:transdeg}
 In addition to the notation of Definitions \ref{defi:sPVring} and \ref{def:spgaloisgroup}, assume that $L_A$ is a field. Then $\sdim(\sGal(L_A|K))=\strdeg(L_A|K)$.  
 \end{prop}
\begin{proof}
	It is shown in \cite[Lemma 2.53]{OvWib} that $\sdim(\sGal(L_A|K))=\sdim(S_A)$. By \cite[Proposition 3.1]{Wibmer:OnTheDimension}, the $\s$-dimension of a finitely $\s$-generated $K$-$\s$-algebra, that is an integral domain with $\s$ injective, agrees with the $\s$-transcendence degree of its field of fractions. Therefore $\sdim(S_A)=\strdeg(L_A/K)$.
\end{proof}
 
The following proposition explains the connection between the $\s$-Galois theory of linear difference equations and the (usual) Galois theory of linear difference equations. The meaning of ``Zariski dense $\s$-closed subgroup'' is explained in Definitions \ref{defi: sclosed subgroup} and \ref{defi: Zariski dense}.
 
\begin{prop}
	\label{propo:schematicalgebraicgaloisgroupcomparaison}
 In addition to the notation of Definitions \ref{defi:sPVring} and \ref{def:spgaloisgroup}, assume that $L_A$ is a field.
Then $K_A=K(U)\subset L_A$ is a Picard-Vessiot extension for $\f(Y)=AY$
and the $\s$-Galois group $G=\sGal(L_A|K)$ is a Zariski dense $\s$-closed subgroup of the Galois group $\G=\Gal(K_A|K)$.
\end{prop}
\begin{proof}
Because $L_A^\f=K$, it is clear that $K_A$ is a Picard-Vessiot extension for $\f(Y)=AY$. For every $C$-$\s$-algebra $B$, every $K\otimes_C B$-$\f\s$-automorphism of $S_A\otimes_C B$, restricts to a $K\otimes_C B$-$\f$-automorphism of $R_A\otimes_C B$. So $G(B)\subset \G(B)$ and we obtain an inclusion $G\subset [\s]_C\G$, where $[\s]_C\G$ is defined in Example \ref{ex1}. As we can see in the paragraph after Definition \ref{defi: Zariski dense}, to prove that $G$ is Zariski dense in $\G$, it suffices to prove that the corresponding map $C[\G]\to C\{G\}$ of coordinate rings is injective. But this map is the inclusion $(R_A\otimes_K R_A)^\f\to (S_A\otimes_K S_A)^\f$, which is clearly injective.
\end{proof}

 \section{Auxiliary results}\label{sec: aux}
 
 In this section, we gather some auxiliary results needed for the proof of Theorem \ref{thm: main}. 
 
\subsection{Extending the constants} 
We will need the following elementary lemma about solutions of linear difference equations.

\begin{lemma} \label{lemma: fsolutions and base extension}
	Let $K$ be a $\f$-field, $C=K^\f$, $R$ a $K$-$\f$-algebra, and $A\in\GL_n(K)$. Then, for every $C$-algebra $B$ (considered as a $\f$-constant $\f$-ring) one has
	$$\{y\in R^n \ | \ \f(y)=A y\}\otimes_C B=\{y\in (R\otimes_C B)^n \ |\ \f(y)=A y \}.
	$$  
\end{lemma}
\begin{proof}
	Clearly $\{y\in R^n \ | \ \f(y)=A y\}\otimes_C B\subset \{y\in (R\otimes_C B)^n \ |\ \f(y)=A y \}$. Let $y$ be contained in the right-hand side and fix a $C$-basis $(b_i)_{i\in I}$ of $B$. Then $y=\sum y_i\otimes b_i$ for some uniquely determined $y_i\in R^n$. We have
	$$\sum \f(y_i)\otimes b_i=\f(y)=Ay=\sum Ay_i\otimes b_i.$$
	Therefore $\f(y_i)=Ay_i$ and so $y$ is contained in the left-hand side.
\end{proof}

\subsection{Compatible difference systems}

Two difference systems 
$\phi(Y)=AY$ and $\sigma(Y)=\widetilde{A}Y$ over a $\f\s$-field $K$ are \emph{compatible} if
\begin{equation}\label{eq:compatibilitycondition}
\phi(\widetilde{A})A =\sigma(A)\widetilde{A}\,.
\end{equation}
This condition is equivalent to the existence of $K$-$\sphi$-algebra $S$ and $U\in\GL_n(S)$ such that
$\phi(U)=AU$ and $\sigma(U)=\widetilde{A}U$. 
Compatible difference systems over the projective line have been studied 
by Sch\"{a}fke and Singer in \cite{SchaefkeSinger}. 
The following result shows that compatibility is 
a strong constraint for  discrete systems over the projective line. 

\begin{prop}[Theorem 13 in \cite{SchaefkeSinger}]\label{prop:compatibilityimpliesconstantcoeff}
Let $K$ be one of the $\sphi$-fields defined in Cases \textbf{2S}, \textbf{2Q}, and \textbf{2M}. 
Let $A,\widetilde{A} \in \GL_n(K)$ be such that the systems $\phi(Y)=AY$ and $\sigma(Y)=\widetilde{A}Y$ are compatible. 
Then, there exists 
$T \in \GL_n(K)$ such that $\phi(T)AT^{-1} \in \GL_n(\C)$ and $\sigma(T)\widetilde{A}T^{-1} \in \GL_n(\C)$. 
In other words, the systems $\phi(Y)=AY$ and $\sigma(Y)=\widetilde{A}Y$ are simultaneously  equivalent over $K$ 
to difference systems with constant coefficients. 
\end{prop}

 \subsection{Reducible Galois groups}
 
For a finite dimensional $C$-vector space $V$, a closed subgroup $\G$ of $\GL_V$ is \emph{irreducible} if $V$ is an irreducible representation of $\G$. If this is not the case, $\G$ is called \emph{reducible}.

The following lemma characterizes linear difference equations with a reducible Galois group. 

\begin{lemma}[Lemma 4.4 in \cite{ADH}]\label{lem: reducible1}
Let $K$ be a $\phi$-field with $C=K^\phi$ algebraically closed. Let $A \in \GL_n(K)$ and 
 let $\G\subset \GL_{n,C}$ be the Galois group of $\phi(Y)=AY$.
For an integer $r$ with $0<r<n$ the following statements are equivalent.
\begin{itemize}
\item There exists a $\G$-subrepresentation of $ C^n$
 of dimension $r$ over $C$.

\item There exists $T \in \GL_n(K)$ such that
 $$\phi(T)AT^{-1} =\begin{pmatrix}
B_1 &B_2 \\
0 & B_3
\end{pmatrix}$$ with $B_1 \in \GL_r(K)$.
\end{itemize}
In particular, $\G$ is reducible if and only if the above statements hold for some $r$, with $0<r<n$.  
\end{lemma}

\subsection{Condition $\cH$ and consequences}
Similarly to \cite{ADH}, we need to consider the following setting. 
We say that a $\sphi$-field  $K$  satisfies \emph{Condition $\cH$} if the following properties hold. 

\begin{itemize}
\item $\s\colon K\to K$ is an automorphism.	
\item $C=K^\phi$ is algebraically closed. 
\item  For every positive integer $r$, $K$ has no finite nontrivial $\f^r$-field  extension. 
\end{itemize}

\begin{rmk}\label{rmk:iteratetheoperatorConditionH}
Note that if $K^{\phi}=C$ is an algebraically closed field then by \cite[Lemma 4.8]{ADH}, $K^{\phi^s}=C$ 
for all positive integers $s$. Thus, any $\sphi$-field satisfying Condition $\cH$, also satisfies Condition $\cH$ as 
a $(\phi^r,\sigma^s)$-field, for all positive integers $r$ and $s$. 
\end{rmk}

\begin{lemma}[Lemma 4.9 in \cite{ADH}] \label{lemma: base fields satisfy H}
	Let $K$ be one of the $\fs$-fields defined in Cases \textbf{2S}, \textbf{2Q}, and \textbf{2M}. Then $(K,\phi)$ satisfies Condition $\cH$.
\end{lemma}

The following lemma is a straightforward generalization of \cite[Lemma 4.5]{DHRqdiffhypergo}.

\begin{lemma}  \label{lem:relativalgclosedbasefieldPPVfield}
Let $K$ be a $\sphi$-field satisfying Condition $\cH$. 
Let  $A \in \GL_n(K)$ and let $L_A$ be a $\s$-Picard-Vessiot extension for $\phi(Y)=AY$  over $K$. 
If $L_A$ is a field, then $K$ is relatively algebraically closed in $L_A$.
\end{lemma}

\begin{proof}
	Let $M$ denote the relative algebraic closure of $K$ in $L_A$. We have to show that $M=K$. 
	Given a fundamental  matrix $U\in\GL_n(L_A)$ for $\f(Y)=AY$ and an integer $i\geq 0$, 
	the intersection $M\cap K(U,\ldots,\s^i(U))$ 
	 is a finite $\f$-field extension of $K$ and therefore trivial. Thus $M=K$.
\end{proof}

We are now interested in iterating difference systems. 
 Given $A \in \GL_n( K)$ and a positive integer $r$, we set $A_{[r]}:=\phi^{r-1}(A) \cdots \phi(A) A$. 
 Note that $\phi(Y)=AY$ implies that $\phi^{r}(Y)=A_{[r]}Y$.
The following proposition shows that, considering some iterates of the operators $\s$ and $\phi$ if necessary, 
one can always reduce the situation to the case where  the  $\s$-Picard-Vessiot extension $L_A$ is a field. 

 \begin{prop}[Proposition 4.6 in \cite{DHRqdiffhypergo}]\label{propo:existenceparamfieldsolutionqdiff}
Let $K$ be a $\sphi$-field such that $C=K^\phi$ is algebraically closed. Let $A \in \GL_n( K)$, $F$
a $\phi\s$-field extension of $K$ such that $F^\phi=K^\phi$, and let 
$u_1,\dots,u_n \in F^*$ be such that $(u_1,\dots,u_n)^{\top}$ is a solution to $\phi(Y)=AY$. 
Then the following properties hold. 

\begin{itemize}
\item[{\rm (1)}] 
There exist  some positive integers $r$ and $s$ and  a $\s^s$-Picard-Vessiot  
extension $L_A$ for  the system $\phi^r(Y)=A_{[r]}Y$  over $(K,\phi^r)$ such that $L_A$ is a field and  
$K\{u_1,\ldots,u_n\}_{\s^s,\f^r}$ embeds into $L_A$. 

\item[{\rm (2)}]  If $K$ satisfies Condition $\cH$, then the  Galois group of $\phi^{r}(Y)=A_{[r]}Y$ over $(K,\phi^r)$ 
coincides with the connected component of the Galois group of the system ${\phi(Y)=AY}$ over $(K,\phi)$.  
\end{itemize}
  \end{prop}
  
 \begin{proof}
 The paper \cite{DHRqdiffhypergo} is concerned with Case $\textbf{2Q}$ but the proof works 
 the same at this level of generality.
 \end{proof}

\begin{rmk}\label{rem:groupofiteratesifconnected}
As a straightforward consequence of  (2) in Proposition \ref{propo:existenceparamfieldsolutionqdiff}, 
we obtain that if  the Galois group $\G$ of the system ${\phi(Y)=AY}$ over $(K,\phi)$ is connected, 
then, for every positive integer $r$, the Galois group of 
$\phi^r (Y)=A_{[r]} Y$ over $(K,\phi^{r})$ coincides with $\G$. 
\end{rmk}

A $\s$-algebraic group $G$ over a $\s$-field $C$ is \emph{$\s$-integral} if its coordinate ring $C\{G\}$ is a \emph{$\s$-domain}, i.e. $C\{G\}$ is an integral domain and $\s\colon C\{G\}\to C\{G\}$ is injective.

\begin{lemma}\label{lemma:relatalgclosedbasefieldsigmaintegralPPVgroup}
	Let $K$ be a $\sphi$-field satisfying Condition $\cH$
	and let $L_A |K$ be a $\s$-Picard-Vessiot extension for $\phi(Y)=AY$. 
	If $L_A$ is a field, then $\sGal(L_A|K)$ is $\sigma$-integral and the Galois group of $\f(Y)=AY$ is connected.
\end{lemma}

\begin{proof}
	Since $K$ is an inversive $\s$-field, i.e. $\s\colon K\to K$ is surjective, by \cite[Proposition 1.2  (iii)]{TomasicWibmer:Babbitt}, the $K$-$\s$-algebra $L_A$ satisfies all the equivalent conditions of \cite[Proposition 1.2]{TomasicWibmer:Babbitt}. In particular, point (vii) of that proposition implies that
	${\s\colon L_A\otimes_K S'\to L_A\otimes _K S'}$ is injective for any $K$-$\s$-algebra $S'$ with $\s\colon S'\to S'$ injective. Thus the map $\s\colon L_A\otimes_K L_A\to L_A\otimes_K L_A$ is injective. Since $K$ is relatively algebraically closed in $L_A$ by Lemma \ref{lem:relativalgclosedbasefieldPPVfield}, we see that $L_A$ is a regular field extension of $K$. So by \cite[Chapter V, \S 17, No. 3, Proposition 2]{Bourbaki}, $L_A\otimes_K L_A$ is an integral domain.
	
	Let $S_A\subset L_A$ denote the corresponding $\s$-Picard-Vessiot ring. Because $L_A\otimes_K L_A$ is a $\s$-domain we see that also $S_A\otimes_L S_A\subset L_A\otimes_K L_A$ is a $\s$-domain. By (\ref{eq:torsor})
	we have $S_A\otimes_K S_A=S_A\otimes_C C\{G\}$, where $G=\sGal(L_A|K)$. Therefore also $C\{G\}$ is a $\s$-domain and so $G$ is $\s$-integral. Because the coordinate ring $C[\G]$ of the Galois group $\G$ of $\f(Y)=AY$ is contained in $C\{G\}$ (cf. the proof of Proposition \ref{propo:schematicalgebraicgaloisgroupcomparaison}), it follows that $C[\G]$ is an integral domain. Therefore $\G$ is connected (cf. \cite[Summary 1.36]{Milne:algebraicGroups}).	
\end{proof} 

 \section{Proof of Theorem \ref{thm: main}}\label{sec: main}
 
  Throughout this section the $\sphi$-fields $K$ and $F$ are as in one of Cases \textbf{2S}, \textbf{2Q}, 
  and \textbf{2M}. Our main goal is to establish the following result. 

\begin{thm}\label{thm: sigmatranscendence}
Let $f \in F$ be a solution of the linear $\phi$-equation of order $n$
\begin{equation}\label{eq: difference}
 \phi^{n} (y) + a_{n-1}\phi^{n-1}(y)+\cdots + a_{1} \phi(y)+a_{0} y=0\,,
\end{equation}
where $a_0,\ldots,a_{n-1}\in K$. Then either $f \in K$ or $f$ is $\sigma$-transcendental over $K$.
\end{thm}

  The following simple lemma shows that Theorem \ref{thm: sigmatranscendence} implies Theorem \ref{thm: main}.

\begin{lemma}\label{lemma:algebraicindependenceversussigmaalgebraicdependence}
	Let $L|K$ be a $\s$-field extension. Let $f,g \in L$ such that $g$ is a nonzero solution of a linear $\sigma$-equation over $K$. If $f$ and $g$ are algebraically dependent over $K$, then $f$ is $\sigma$-algebraic over $K$. 
\end{lemma}

\begin{proof}
 By assumption $g$ is $\s$-algebraic over $K$ and because $f$ and $g$ are algebraically dependent, $f$ is $\s$-algebraic over $K\langle g\rangle_\s$. By the transitivity of being $\s$-algebraic (\cite[Theorem 4.1.2 (i)]{Levin}),
 $f$ is $\s$-algebraic over $K$. 
\end{proof} 
 
\begin{ex} 
Let us equip $\C(x)$ with two automorphisms $\rho (x)=x+1$ and $\s (x)=x+h$, with $h \in \C $. 
When $h\notin \Q$, the $\s$-transcendence  of the Gamma function has been proved in \cite[Example 3.7]{OvWib}. By Lemma \ref{lemma:algebraicindependenceversussigmaalgebraicdependence}, we then deduce that for $h\notin \Q$, $\Gamma(x)$ and $\Gamma(xh^{-1})$ are algebraic independent   over $\C(x)$.
\end{ex}

Similarly to \cite{ADH}, Theorem \ref{thm: sigmatranscendence} is proved by  
induction on the order $n$ of Equation \eqref{eq: difference}.  
Before proving Theorem~\ref{thm: sigmatranscendence} in full generality, we first consider the following 
special cases. 
\begin{itemize}
\item[-] The function $f$ is a solution of an inhomogeneous equation of order one. 
\item[-] The function $f$ is a solution of a difference equation whose corresponding Galois group is both connected and irreducible.  
\end{itemize}


\subsection{Affine order one equations}\label{sec:affinediffequ} 

We begin this subsection by studying homogeneous order one equations $\f(y)=ay$. The following proposition characterizes the coefficients $a\in K$ for which this equation has a nontrivial $\s$-algebraic solution.  
Case $\textbf{2Q}$ is already proved  in \cite[Proposition 5.3]{DHRqdiffhypergo}. The following proposition summarizes all three cases.

\begin{prop}\label{propo:sprang1}
Let $K$ and $(\phi,\s)$ be defined as in Cases \textbf{2S}, \textbf{2Q}, and \textbf{2M}.   Let 
 $a \in K^*$. Let  $L$ be a $\s$-Picard-Vessiot extension for  $\phi(y)=ay$ and assume that $L$ is a field.
Let $u \in L^*$ be such that 
$\phi(u) = a u $. 
 Then the following statements are equivalent. 
\begin{enumerate}
\item The element $u\in L$ is $\s$-algebraic over $K$.
\item There exist $c \in \C^{*}$, $n\in \Q$ ($n=0$ in Cases $\textbf{2S}$ and $\textbf{2M}$) and $b \in K^{*}$ such that $a=cx^{n}\frac{\phi(b)}{b}$.
\end{enumerate}

\end{prop}
 \begin{proof}
If (2) holds,  a straightforward computation, similar to the one in the proof of \cite[Proposition 5.3]{DHRqdiffhypergo}, yields $\s\left( \frac{\s(u/b)}{u/b} \right)/ \left(\frac{\s(u/b)}{u/b}\right) \in L^\phi=\C$. This shows that $u$ is $\s$-algebraic over $K$. So $(2)$ implies $(1)$. 

Let us assume that $u$ is  $\sigma$-algebraic over $K$.  The Galois theoretic arguments used in   \cite[Theorem~3.1]{OvWib} for Cases \textbf{2Q} and  \textbf{2S} also apply to Case \textbf{2M} and we find that, since  $u$ is  $\sigma$-algebraic over $K$, there exist
$r\in \Z_{\geq 0}$, $k_{0},\dots,k_{r} \in \Z$, with $k_{0}k_{r}\neq 0$, and $b \in K^*$, such that 
\begin{equation}\label{eq99}
\displaystyle \prod_{\ell=0}^{r}\s^{\ell}(a^{k_{\ell}})=\frac{\phi (b)}{b} \cdot
\end{equation}

Consider first Case $\textbf{2S}$. 
 Following  \cite[\S 4.1]{hardouin2008hypertranscendance}, we define the $h_{1}$-divisor $\div_{h_{1}}(f)$ of $f \in \C(x)$ as the formal sum $\div_{h_{1}}(f) =\sum_{[\alpha] \in \C/h_{1}\Z} n_\alpha[\alpha]$, where $n_\alpha$ is the sum of the valuations of $f$ at the points $[\alpha]=\alpha + h_{1}\Z$. By \cite[Lemme 4.3]{hardouin2008hypertranscendance}, there exists $c \in \C^*$ and $0\neq b \in \C(x)$ such that $f=c \frac{\phi(b)}{b}$ if and only if $\div_{h_{1}}(f)=0$. 
Thanks to this characterization, we only need to prove that $\div_{h_{1}}(a)=0$. Suppose to the contrary that this is not the case and that $\div_{h_{1}}(a)=\sum_{i=1}^m n_i [\zeta_i]$, for some pairwise distinct elements $[\zeta_i] \in \C/h_{1}\Z$ and nonzero integers $n_i$. From \eqref{eq99}, we deduce
\begin{equation}\label{eq:bla}
0=\div_{h_{1}}\left(\frac{\phi(b)}{b}\right)=\sum_{l=0}^{r}k_l\sum_{i=1}^m n_i[\zeta_i -h_{2} l ]. 
\end{equation}
Let 
$$
I=\{i \in \{1,\ldots,m\} \ \vert \  [\zeta_{i}]=[ \zeta_{1} +h_{2}n] \mbox{ for some integer } n \}.
$$ 
Let $i_{1},\ldots,i_{s}$ be pairwise distinct integers such that $I=\{i_{1},\ldots,i_{s}\}$. Up to renumbering, we can assume that 
$$
[\zeta_{i_{1}}] \prec \cdots \prec [\zeta_{i_{s}}]
$$
where, for any $[x],[y] \in \C/h_{1}\Z$, $[x] \prec [y]$ means that $[y] = [x] +kh_{2}(=[x +kh_{2}])$ for some $k \in \N^{*}$. Note that the binary relation  $\prec$ is well defined because $\frac{h_{1}}{h_{2}} \notin \Q$. It is clear that   $[\zeta_{i_{1}} -rh_{2}] \prec  [  \zeta_{i_{k}} -jh_{2}]$ for all $j \in \{0,\ldots,r\}$ and $k \in \{1,\ldots,s\}$ such that $(j,k)\neq (r,1)$.  
Moreover,  for $j \in \{0,\ldots,r\}$ and $i \in \{1,\ldots,m\}\setminus I$, we have $[ \zeta_{i_{1}} -r h_{2}] \not=[\zeta_{i} -jh_{2}]$ by definition of $I$.  Therefore, the coefficient of $[\zeta_{i_{1}} -rh_{2}]$ in equation (\ref{eq:bla}) is equal to $0$, {\it i.e.} $k_{r}n_{i_{1}}=0$. This provides a contradiction.\par 
 
Consider now Case \textbf{2M}. For $f\in F^*$, let $v_{f}\in \Q$ be its valuation and 
let $t_{f}=fx^{-v_{f}}|_{x=0}\in \C^*$.
By \cite[Lemma~20]{Ro15}, the equation $\phi(y) =\frac{ax^{-v_{a}}}{t_{a}} y $ has a 
solution $f\in F^*$. This can also be verified directly by recursively solving for the coefficients of $f$.

  Let $\tilde{f}:= \prod_{\ell=0}^{r}\s^{\ell}(f^{k_{\ell}}) \in F^*$ and $\tilde{c}:= \prod_{\ell=0}^{r}t_{a}^{k_{\ell}}\in \C^*$. With \eqref{eq99} and $\phi\sigma=\sigma\phi$, we find that there exists $N\in \Z$ such that
\begin{equation}\label{eq4}
\phi (\tilde{f})=x^{N}\displaystyle \prod_{\ell=0}^{r}\frac{\s^{\ell}(a^{k_{\ell}})}{t_{a}^{k_{\ell}}}\tilde{f}
=x^{N}\frac{\phi(b)}{b}\frac{\tilde{f}}{\widetilde{c}}=\frac{\phi(x^{N/(p_1-1)}b)}{x^{N/(p_1-1)}b}\frac{\tilde{f}}{\widetilde{c}}\cdot
\end{equation}
 For $g\in F^*$, we have $t_{\phi(g)}=t_{g}$.
By \eqref{eq4}, we obtain $\tilde{c}=1$.  Hence $\dfrac{\tilde{f}}{x^{N/(p_1-1)}b} \in F^\phi=\C$ and there exists $d \in \C^*$ such that 
\begin{equation}\label{eq100}
\displaystyle \prod_{\ell=0}^{r}\s^{\ell}(f^{k_{\ell}})=dx^{N/(p_1-1)}b.
\end{equation}
Let $\partial$ be the derivation $x\frac{d}{dx}$. Since 
$\partial \phi=p_{1} \phi \partial$ and $\partial \s^\ell=p_{2}^\ell \s^\ell \partial$, 
we  can easily compute the logarithmic derivatives of 
$\phi(f) =\frac{ax^{-v_{a}}}{t_{a}} f $ and 
 of  \eqref{eq100}. We obtain that $\partial f/f\in F$ satisfies 
the following linear $\s$-difference and $\phi$-difference equations over $K$:
$$
p_{1} \phi(y)= \frac{\partial \left( ax^{-v_{a}}\right)}{ax^{-v_{a}}} + y \;\;\mbox{ and }\;\;  
\sum_{\ell=0}^r p_{2}^{\ell} k_\ell \s^\ell(y)=\frac{\partial (x^{N/(p_1-1)}b)}{x^{N/(p_1-1)}b}\cdot
$$
 By Theorem~\ref{lem:linearlyclosedfieldextension}, $\partial f/f\in K$. So $f\in F$ satisfies a linear $\phi$-difference equation and a linear differential equation with coefficients in $K$.  By \cite[Theorem~1.2]{ADH} (see also \cite{SchaefkeSinger}), $f\in K$. Therefore 
 $$a =t_a \frac{\phi(fx^{v_{a}/(p_{1}-1)})}{fx^{v_{a}/(p_{1}-1)}}$$
 has the desired form. 
\end{proof}

In \cite[\S 3.3]{DVHaWib2}, the authors study  the difference algebraic relations satisfied by the solutions of a differential equation of the form $y'=ay+b$ , with  $ a, b \in K$. Here, we adapt and generalize some of the arguments from \cite{DVHaWib2} to our context in order to establish the following special case of Theorem \ref{thm: sigmatranscendence}

\begin{prop} \label{propo:shift_shift_algrank1}Let $K$, $F$, and $(\phi,\s)$ be defined as in Cases \textbf{2S}, \textbf{2Q}, and \textbf{2M}.  Let $f\in F$ satisfy ${\phi(f) = a f + b}$ with $ a, b \in K$. Then either $f$ is $\s$-transcendental over $K$, or $f\in K$.
\end{prop}

Recall that $\phi$ is an automorphism of $K$. Therefore, the result is trivial if $a=0$. So let us assume that  $a\neq 0$ and let us  consider the $\phi$-system 
\begin{equation}\label{eq2}
\phi  (Y)=\begin{pmatrix}
a&b\\
0&1
\end{pmatrix}Y=AY.
\end{equation}
Then, $\begin{pmatrix} f \\ 1\end{pmatrix}$ is a vector solution of \eqref{eq2}. Recall that $F^{\phi}=K^{\phi}$.
By  Proposition \ref{propo:existenceparamfieldsolutionqdiff},  
there exist $r,s \in \N$ and a $\s^s$-Picard-Vessiot extension $L_{A}$ over $(K,\phi^{r})$ for the system $\phi^r(Y)=A_{[r]}Y$ with $A_{[r]}=\phi^{r-1}(A) \cdots A$ such  that the following properties hold. 

\begin{itemize}
 \item $L_{A}$ is a field.
 \item $f\in L_{A}$.
 \item Let $K_{A}\subset L_{A}$ denote the Picard-Vessiot extension as in Proposition~\ref{propo:schematicalgebraicgaloisgroupcomparaison}. Then 
 the classical Galois group $\Gal(K_{A} | K)$ of $\phi^r(Y)=A_{[r]} Y$ over $(K,\phi^{r})$ coincides with the identity component of the classical Galois group  of ${\phi(Y)=A Y}$ over $(K,\phi)$.
 \end{itemize}
Note that the equation corresponding to $\phi^r(Y)=A_{[r]}Y$ is still an affine order one equation. Without loss of generality, we may assume that  $r=1$. By Lemma \ref{lem2},  without loss of generality, up to replacing $\s$ by some  power of $\s$, we may also reduce to the case where $s=1$.\par 
 We may assume that the fundamental matrix is of the form
$U=\begin{pmatrix}
u&f\\
0&1
\end{pmatrix}$, where $0\neq u\in L_{A}$ is a solution of $\phi (u)=au$, since the lower left entry of $U$ can always be eliminated by subtracting a multiple of  $\begin{pmatrix} f \\ 1\end{pmatrix}$ from the first column. Let $\sGal( L_{A}|K)$ be the $\s$-Galois group. 
 Via its action on the fundamental matrix $U$, for any $\C$-$\s$-algebra $B$,  $\sGal( L_{A}|K)(B)$  can be represented as a subgroup of
$$
\left. \left\{ \begin{pmatrix}
\alpha &\beta\\
0&1
\end{pmatrix}\right| \alpha\in B^{*},\beta \in B \right\}.
$$ 
Let $\mathbb{G}_u$ denote the algebraic subgroup of $\GL_{2,\C}$ given by
$$
\mathbb{G}_u(B)=\left\{ \begin{pmatrix}1 & \beta \\ 0 &1 \end{pmatrix} \Big|\ \beta \in B \right\},
\hbox{~for any $\C$-$\s$-algebra $B$,}
$$
and set $G_u= \sGal( L_{A}|K)  \cap\mathbb{G}_u$. Since $\tau(u\otimes 1)= u \otimes \alpha$, for $\tau=\begin{pmatrix}
\alpha &\beta\\
0&1
\end{pmatrix}\in \sGal( L_{A}|K)(B)$,  the Galois correspondence implies that
$G_u=\sGal(L_{A}|K\langle u\rangle_\s)$.
Moreover, we have $\phi(f/u)= f/u+ b/(au)$, so that $L_{A}|K\langle u\rangle_\s$ is a
$\s$-Picard-Vessiot extension for $\phi(y)= y + b/(au)$.
The action of an element $\tau=\begin{pmatrix}1 & \beta \\ 0 &1 \end{pmatrix}\in G_u(B)=\sGal(L_{A}|K\langle u\rangle_\s)(B)$ on $L_{A}$ is given by
\begin{equation}\label{eq3}
\tau\l(\frac{f}{u} \otimes 1 \r)=\frac{f}{u}\otimes 1+ 1\otimes\beta. 
\end{equation}
The situation is summarized in the following picture:
$$
\xymatrix{  L_{A} \ar@{-}[d]  \ar@/_1pc/@{-}[d]_{G_u}  \ar@/^3pc/@{-}[dd]^{\sGal( L_{A}|K)} \\
K\langle u \rangle_\s \ar@{-}[d]   \\
K }
$$

The proof of   \cite[Proposition 3.16]{DVHaWib2} for nonhomogeneous differential equations of order one with discrete parameters passes 
word for word to our context and yields the following lemma. 
 \begin{lemma}\label{lem:homogensigmatransnonhomogtoo}
 Let $L|K$ be a $\sphi$-field extension such that $C=K^\phi=L^\phi$. We fix $u,f \in L, u \neq 0$, such that $\phi(f)=af +b$ and $\phi(u)=au$ with $a,b \in K$. Assume that $f \notin K$ and that $\s\colon K\to K$ is surjective. If $u$ is $\s$-transcendental over $K$, then 
 $f$ is $\s$-transcendental over $K$.
 \end{lemma}
 
We will also need the following auxiliary result that generalizes \cite[Lemma 6.5]{HS}. 

\begin{lemma}\label{claim:descentsolutionnonhomogen}
For $\tilde{a},\tilde{b} \in K$, assume that there exists a $\s$-Picard-Vessiot extension $L_A$ for   $ \phi(Y)=  \begin{pmatrix}
\tilde{a}&\tilde{b}\\
0&1
\end{pmatrix}Y=AY $ over $K$ that is a field. Consider an  intermediate $\phi$-field $L$ with $L_{A}|L|K$ and let $x\in L_A^*$ be
such that 
$\phi(x)/x=\alpha \in L^*$.  If $\phi(y) = \tilde{a} y + \tilde{b}$ has a solution in $L(x)$, then it has a solution in $L$.
\end{lemma}

\begin{proof}[Proof of Lemma \ref{claim:descentsolutionnonhomogen}]
Let $g\in L(x)$ be a solution of the equation $\phi(y) =  \tilde{a} y + \tilde{b}$. 
Let us first assume that $x$ is algebraic over $L$. 
Then $g$ is algebraic over $L$ and the orbit of $g$ under the Galois group $\G$ of $\phi(Y)=AY $ 
over $L$ is finite. 
Let  $g_{1},\dots, g_{\ell}\in L_A$ denote the elements in the $\G$-orbit of $g$ and set $\tilde{g}:=g_{1}+\dots+g_{\ell}$. Since $\tilde{g}$ 
is fixed by $\G$ it belongs to $L$ by Proposition~\ref{prop:Galois correspondence}. 
As $\tilde{a},\tilde{b}\in K\subset L$, we have $\phi(g_{k}) = \tilde{a} g_{k} + \tilde{b}$ for $1\leq k\leq \ell$.  Therefore, 
$\phi(\tilde{g}/\ell) = \tilde{a} \tilde{g}/\ell + \tilde{b}$ with $\tilde{g}/\ell\in L$. This proves the claim in that case. \par 

Now, let us assume that $x$ is transcendental over $L$. We can extend $\phi$  to the field of Laurent series $L((x))$ by setting $\phi(\sum_{k\geq r} a_kx^k)=\sum_{k\geq r} \phi(a_k)\alpha^kx^k$. 
Writing $g=\sum_{k\geq r}a_kx^k$ with $a_k\in L$ and $r\in\mathbb Z$, we get that $\phi(g)=\sum_{k\geq r} \phi(a_k)\alpha^kx^k$. Using the 
uniqueness of the Laurent series expansion and the fact that $\phi(g)=\tilde{a}g+\tilde{b}$, we deduce that $\phi(a_0)=\tilde{a}a_0+\tilde{b}$. 
Since $a_0\in L$, this completes the proof of the lemma. 
\end{proof}

  After these preliminaries, we are ready to prove Proposition \ref{propo:shift_shift_algrank1}. 
  
  \begin{proof}[Proof of Proposition \ref{propo:shift_shift_algrank1}]
  Let us argue by contradiction and suppose that $f$ is $\s$-algebraic and does not belong to $K$. By Lemma \ref{lem:homogensigmatransnonhomogtoo}, the element $u$ is $\s$-algebraic over $K$. By Proposition \ref{propo:sprang1}, there exist $c \in \C^*,\tilde{b} \in K^*$ and $n \in \Q$, which is zero in Cases $\textbf{2S}$ and $\textbf{2M}$, such that $a =c x^n \phi(\tilde{b})/\tilde{b}$. Replacing $f$ by $f/\tilde{b}$ and $b$ by $b/\f(\widetilde{b})$, we can assume that $\tilde{b}=1$.   By Proposition~\ref{prop:transdeg}, $\strdeg ( L_{A} |K\langle u\rangle_\s)=0 =\sdim (G_u)$,  and the $\s$-algebraic group $G_u$ must be a proper subgroup of $\mathbb{G}_u$ whose $\s$-dimension is $1$. Let $\C[\s]$ denote the set of linear $\s$-operators with coefficients in $\C$.  By \cite[Corollary~A.3]{DVHaWib1}, there exists $\mathcal{L}\in \C[\s]$, such that for $B$ a $\C$-$\s$ algebra, $$G_u (B)= \left\{ \begin{pmatrix}
  1 & \beta \\
  0& 1  \end{pmatrix}\in\GL_n(B)\ \Big|\  \mathcal{L}(\beta)=0\right\}.$$
  
Let us first consider the case $n=0$.   As a consequence of \eqref{eq3}, for $\tau\in\sGal(L_{A}|K\langle u\rangle_\s)(B)=G_u(B)$,
$$\tau\l(\mathcal{L}\l(\frac{f}{u}\r) \otimes 1 \r)=\mathcal{L}\l(\tau\l(\frac{f}{u} \otimes 1\r) \r)=\mathcal{L}\l(\frac{f}{u}\r)\otimes 1+ 1\otimes \mathcal{L}(\beta)=\mathcal{L}\l(\frac{f}{u}\r) \otimes 1.$$
So the Galois correspondence implies that $\mathcal{L}(\frac{f}{u}) \in K\langle u\rangle_\s$. Since $a\in\C^*$, we find that 
$$
\phi(\frac{\s(u)}{u})=\frac{\s(\phi(u))}{\phi (u)}=\frac{\s(a)\s(u)}{au}=\frac{\s(u)}{u} \cdot
$$ 
Then $d:=\frac{\s(u)}{u} \in K^\phi=\C$, so that $K\langle u\rangle_\s=K(u)$. 
Since $\s(u) =d u$ with $d\neq 0$ and $\mathcal{L}(\frac{f}{u}) \in K(u)$, there exists a nonzero
$\widetilde{\mathcal{L}}\in \C[\s]$ such that 
$\widetilde{\mathcal{L}}(f) \in K(u)$. Recall that $a\in \C^*$ so that $\widetilde{\mathcal{L}}(f)$ is a solution of $\phi(y)=ay +\widetilde{\mathcal{L}}(b)$. By Lemma \ref{claim:descentsolutionnonhomogen}, there exists $g \in K$ such that $\phi(g)=ag+\widetilde{\mathcal{L}}(b)$. Then $\frac{\widetilde{\mathcal{L}}(f)-g}{u}\in K(u)$  
is fixed by $\phi$, and since $K(u)^{\phi}\subset L_{A}^{\phi}=\C$, 
there exists $d' \in \C$ such that $\widetilde{\mathcal{L}}(f)= d'u+g$. Since $u$ and $g$ both satisfy 
some linear $\s$-difference equations over $K$, the same holds for $\widetilde{\mathcal{L}}(f)$. 
Therefore, $f$ satisfies a nontrivial linear $\s$-difference equation over $K$. By Theorem \ref{lem:linearlyclosedfieldextension}, we find  $f \in K$. This contradiction concludes the proof in the case $n= 0$ (and hence in Cases $\textbf{2S}$ and $\textbf{2M}$).

Now, let us consider Case $\textbf{2Q}$ with $n \neq 0$. 
By \cite[Theorem~A.9]{DVHaWib2},   we only have to consider the two following  cases. 
\begin{enumerate}
\item There exist integers $\nu>m\geq 0$ such that $\s^\nu(\alpha)=\s^m(\alpha)$ and $\mathcal{L} \in \C[\s ]$ such that $\mathcal{L}(\beta)=0$ for all $\begin{pmatrix}\alpha & \beta \\ 0 &1 \end{pmatrix}\in \sGal( L_{A}|K)(B)$ and all $\C$-$\s$-algebras $B$.
\item There exists an integer $\nu\geq 0$ such that $\s^\nu(\beta)=0$ for all $\begin{pmatrix}1 & \beta \\ 0 &1 \end{pmatrix}\in G_u(B)$ and all $\C$-$\s$-algebras $B$.
\end{enumerate}
Consider the first case. 
For all $\tau\in\sGal(L_{A}|K)(B)$, we have $\tau(u\otimes 1)= u\otimes \alpha $ with $\s^\nu(\alpha)=\s^m(\alpha)$.  
Then $g=\s^\nu(u)/ \s^m(u)$ satisfies $\tau(g\otimes 1)=g\otimes 1$ and
 we deduce from the Galois correspondence that $g \in K$. Since $g$ is a nonzero solution of 
$\phi(y)=q_{2}^{n(\nu-m)} y$ and a nonzero Puiseux series of the form  $\sum_{ \ell \in \frac{1}{r}\Z} g_\ell x^\ell$, we find that $q_{1}^\ell g_\ell= q_2^{n(\nu-m)} g_\ell$ for any $\ell \in \frac{1}{r}\Z$. Since $q_1$ and $q_2$ are multiplicatively independent, the coefficients $g_\ell$ must all vanish. A contradiction with the fact that $g$ is nonzero. \par 
In the second case, a similar computation and the Galois correspondence ensure that 
$\s^\nu(f) \in K\langle u\rangle_\s$. Recall that $u$ is $\s$-algebraic over $K$. Let $m \in \N$ be such that 
$\s^\nu(f)\in K(u,\s(u),\ldots,\s^m(u))$.
 Since $\phi(\s^k(u))=\s^k(a) \s^k(u)$ and $\phi(\s^\nu(f))=\s^\nu(a)\s^\nu(f)+\s^\nu(b)$, one can apply recursively Lemma \ref{claim:descentsolutionnonhomogen} to find a $K$-rational solution  of $\phi(y)=\s^\nu(a)y+\s^\nu(b)$. Since $K$ is inversive, there exist $g \in K$ such that $\phi(g)=ag+b$ and as in the $n=0$ case,  a complex number $d \in \C$ such that $f=g+d u$. 
Since $f \notin K$, the constant $d$ is nonzero. Then, $u=(f-g)/d$ belongs to  $F^*$. Let $k\in \Q$ be its valuation. Since $\phi(u)=au=cx^{n}u$, we find $k=k+n$, which is a contradiction. This ends the proof. 
  \end{proof}

 \subsection{Connected and irreducible Galois groups}
 As a next step, we consider the case where the Galois group is both connected and irreducible.  The goal of this subsection is to prove the following proposition. 
 
  \begin{prop}\label{prop: irreducible}
 Let $K,F$, and $(\phi,\s)$ be defined as in Cases \textbf{2S}, \textbf{2Q}, and \textbf{2M}.
 Let us assume that $n\geq 2$, and let $f\in F$ be a nonzero solution to Equation \eqref{eq: phi0}, that we are going to consider as a $\phi$-system $\phi(Y)=AY$, where $A$ is the corresponding companion matrix. If  
the Galois group of $\phi(Y)=AY$ over $K$ is
connected and irreducible, then $f$ is $\s$-transcendental over $K$.
 \end{prop}

Throughout this subsection we adopt the following conventions. As previously, 
$K$ is one of the $\f\s$-fields of Cases \textbf{2S}, \textbf{2Q}, and \textbf{2M}. We let $L_A$ denote 
a $\s$-Picard-Vessiot extension for $\phi(Y)=AY$, where $A\in\GL_n(K)$. We define $K_A$ as in Proposition~\ref{propo:schematicalgebraicgaloisgroupcomparaison}.
   As we will see, 
Proposition \ref{prop: irreducible} will be deduced from the following more general statement.
\begin{prop}\label{prop: colonne}
Let $A\in \GL_{n}(K)$ with $n\geq 2$.  Assume that $L_A$ is a field and the Galois group $\Gal(K_A|K)$ is irreducible.
	Then every nonzero solution $u\in L_A^n$ of the system $\f(Y)=AY$ contains at least one coordinate that is 
	$\s$-transcendental over $K$. 
\end{prop}

\begin{proof}[Proof of Proposition \ref{prop: irreducible}]
We argue by contradiction, assuming that $f$ is $\s$-algebraic over $K$. 
With $\phi\s=\s\phi$, we deduce that all coordinates of the vector $(f,\dots, \s^{n-1} (f))^{\top}$ are also $\s$-algebraic over $K$. 
Since $(f,\dots,\rho^{n-1}(f))^{\top}\in F^n$ is nonzero, 
Proposition \ref{propo:existenceparamfieldsolutionqdiff} ensures the existence of positive integers $r,s$ 
and a  $\s^{s}$-Picard-Vessiot  extension $L_{A_{[r]}}$ for $\phi^{r}(Y)=A_{[r]}Y$ over $K$  that is a field,
such that  the vector $(f,\dots,\rho^{n-1}(f))^{\top}$ is 
the first column of a fundamental matrix $U$. Furthermore,  Remark \ref{rem:groupofiteratesifconnected}  ensures that 
the Galois group of $\phi^r(Y)=A_{[r]}(Y)$ over $K$ 
is equal to $\Gal(K_{A}|K)$, for the latter is connected. By Lemma \ref{lem2}, all coordinates of the vector $(f,\dots, \s^{n-1} (f))^{\top}$ are also $\s^{s}$-algebraic over $K$. 
     Thus, 
Proposition~\ref{prop: colonne} applies with $\phi$ replaced by $\phi^r$, and provides a contradiction. 
\end{proof}
 
Before proving Proposition~\ref{prop: colonne}, let us first recall some terminology from the theory of linear algebraic groups. 
The \emph{radical} of a linear algebraic group (over an algebraically closed field of characteristic zero) is the largest connected solvable normal closed subgroup. 
A linear algebraic group is \emph{semisimple} if it is connected and its radical is trivial. 
It is \emph{almost-simple} if it is nontrivial, semisimple and every proper normal closed subgroup is finite. 
A linear algebraic group is \emph{simple} if it is nontrivial, semisimple an every proper normal closed subgroup 
is trivial. In particular, a simple group cannot be abelian. 

 Let $A\in\GL_n(K)$ and $d\in \N^*$. We say that the system $\f(Y)=AY$ is \emph{$\s^d$-isomonodromic} if there exists $\widetilde{A}\in\GL_n(K)$ 
 such that $\f(Y)=AY$ and $\s^d(Y)=\widetilde{A}Y$ are compatible, i.e. $\f(\widetilde{A})A=\s^d(A)\widetilde{A}$.
 The following proposition is analogous to \cite[Theorem 6.4]{DVHaWib2} since, by Lemma~\ref{lem:relativalgclosedbasefieldPPVfield}, the base field $K$ is relatively algebraically 
 closed in $L_A$. 
  
 \begin{prop}\label{propo:caracgalspisomono}
Let $A\in \GL_{n}(K)$ with $n\geq 2$.  Assume that $L_A$ is a field and $\Gal(K_A|K)$ is a simple linear algebraic group. 
  If $\sGal(L_A|K)$ is a proper $\s$-closed subgroup of $\Gal(K_A|K)$, 
  then $\f(Y)=AY$ is $\s^d$-isomonodromic for some $d\geq 1$. 
 \end{prop}

\begin{proof}
Let us fix a fundamental matrix $U\in \GL_n(L_A)$ to obtain embeddings of 
$G=\sGal(L_A|K)$ and $\G=\Gal(K_A|K)$ into $\GL_n$. By Proposition~\ref{propo:schematicalgebraicgaloisgroupcomparaison}, $G$ is Zariski dense in $\G$.
We infer from Lemma \ref{lem:relativalgclosedbasefieldPPVfield} that $G$ is $\s$-integral and so, in particular, 
$\s$-reduced, i.e. $\s\colon \C\{G\}\to \C\{G\}$ is injective. Because $K$ satisfies Condition $\cH$, all the assumptions of 
\cite[Theorem A.20]{DVHaWib1} are satisfied. It is shown in the proof of that theorem 
(see the fist displayed formula on page 114), that there exist $h\in \GL_n(\C)$ and 
$d\geq 1$ such that for every $\C$-$\s$-algebra $B$ and every $\tau\in G(B)\leq \GL_n(B)$ 
we have $\s^d(\tau)=h\tau h^{-1}$. It thus follows from \cite[Theorem 2.55]{OvWib} that $\f(Y)=AY$ is 
$\s^d$-isomonodromic.
\end{proof} 

We refer to Example \ref{ex1} for the definition of $[\s]_\C\Gal(K_A|K)$ that is used in the following proposition.

\begin{prop}\label{prop:almostsimplecase}
Let $A\in \GL_{n}(K)$ with $n\geq 2$.  Assume that $L_A$ is a field. If $\Gal(K_A|K)$ is a simple linear algebraic group,  
 then $\sGal(L_A|K)=[\s]_\C\Gal(K_A|K)$. In particular, $\strdeg(L_A|K)>0$.
\end{prop}

\begin{proof} 
We argue by contradiction, assuming that $G=\sGal(L_A|K)$ is properly contained in $[\s]_\C\G$, where $\G=\Gal(K_A|K)$. 
By Proposition 
\ref{propo:caracgalspisomono},  we obtain the existence of a matrix $\widetilde{A} \in \GL_n(K)$ and an integer $d\geq 1$ 
such that the systems ${\phi(Y)=AY}$ and $\s^d(Y)=\widetilde{A}Y$ are compatible. By Proposition \ref{prop:compatibilityimpliesconstantcoeff}, we deduce that the system $\phi(Y)=AY$ is equivalent over $K$  
to  a system $\phi (Y)=A_{1}Y$, where  $A_{1} \in \GL_n(\C)$. By \cite[Lemma~2.1 and Remark 2.2]{ArrecheSinger}, 
the Galois group of such a system is always abelian, providing a contradiction with the assumption that 
$\G$ is simple. Thus $G=[\s]_\C\G$.
 
Now, since $\G$ is a simple linear algebraic group, it cannot be finite and ${\dim(\G)>0}$.  Using Proposition~\ref{prop:transdeg} and Example \ref{ex: dimensions agree}, we thus obtain that 
$\strdeg(L_A|K)=\sdim(G)=\dim(\G)>0$.
 \end{proof}

\begin{prop}\label{prop:generalcase}
Let $A\in \GL_{n}(K)$ with $n\geq 2$.  Assume that $L_A$ is a field.     Then,  either $\Gal(K_A|K)$ is a connected solvable linear algebraic group 
       or $\strdeg(L_A|K)>0$.
\end{prop}

\begin{proof}
	 Lemma \ref{lemma:relatalgclosedbasefieldsigmaintegralPPVgroup} ensures  that $\G=\Gal(K_A|K)$ is connected.
	Let us assume that $\G$ is not solvable and let us show that $\strdeg(L_A|K)>0$.  
	As $\G$ is not solvable, the radical $R(\G)$ of $\G$ is a proper closed normal subgroup of $\G$ and $\G/R(\G)$ is a nontrivial   
	semisimple linear algebraic group. It follows from the structure theory of semisimple linear algebraic groups, that any 
	 nontrivial semisimple linear algebraic groups is an almost-direct product of (finitely many) almost-simple groups (\cite[Theorem 21.51]{Milne:algebraicGroups}).  
	 It follows that such a 
	 group has a simple quotient. Thus, there exists a normal closed subgroup $\mathcal{N}$ of $\G$ 
	such that $\G/\mathcal{N}$ is a simple linear algebraic group. By the second fundamental theorem of Galois 
	theory (Proposition \ref{cor:Galoisgroupconnected}) there exist an integer $n'$, $1\leq n'\leq n$ and 
	$A'\in\GL_{n'}(K)$ such that $(K_A)^{\mathcal{N}}$ is a Picard-Vessiot extension for $\f(Y)=A'Y$ over $K$, 
	with Galois group $\G/\mathcal{N}$. Given a fundamental matrix $U'\in\GL_{n'}((K_A)^\mathcal{N})$ of $\f(Y)=A'Y$, we have that
	$L_{A'}=K\langle U'\rangle_\s\subseteq L_A$ is a $\s$-Picard-Vessiot extension for $\f(Y)=A'Y$ over $K$.  
	Since the Galois group of $\f(Y)=A'Y$ is simple, it follows from 
	Proposition~\ref{prop:almostsimplecase} that 
	$\strdeg(L_{A'}|K)>0$. Hence we also have $\strdeg(L_A|K)>0$, as desired. 
\end{proof}

\begin{proof}[Proof of Proposition \ref{prop: colonne}]
	Let $S_A\subset L_A$ denote the $\s$-Picard-Vessiot ring. As in Appendix \ref{subsec:sigmaalgebraic}, 
	we let $\pi(S_A)=\pi(S_A|K)$ denote the set of all elements of $S_A$ that are $\s$-algebraic over $K$. 
	Since $L_{A}$ is a field, $S_A$ is a domain. By Corollary~\ref{prop1}, $\pi(S_A)$ is a $K$-$\sphi$-subalgebra 
	of $S_A$.         Let us assume by contradiction that all coordinates of $u$ are $\s$-algebraic over $K$.   Then, $u\in \pi(S_A)^n$. 
         Let $V\subseteq (S_A)^n$ denote the $n$-dimensional $\C$-vector space of all solution of $\f(Y)=AY$ in $S_A^n$. 
	Since, by assumption,  $u\in V\cap \pi(S_A)^n=:W$, we see that $W$ is  nonzero subspace of $V$. 
	By Corollary \ref{prop1}, $\pi(S_A)$ is stable under the action of $\sGal(L_A|K)$ and so the same holds for $W$.
	Since the representation $V$ of $\Gal(K_A|K)$ is irreducible, 
	Lemma \ref{lemma:irreduciblesigmarepresentationirreduciblezariskiclosurerep} implies that the representation 
	$V$ of $\sGal(L_A|K)$ is irreducible too. Thus $W=V$ and so $V\subseteq \pi(S_A)^n$. Hence, all entries of 
	all elements of $V$ are $\s$-algebraic over $K$. Since $L_A$ is $\s$-generated over $K$ by all these entries, 
	it follows that $\strdeg(L_A|K)=0$. Thus, by  Proposition~\ref{prop:generalcase}, $\Gal(K_A|K)$ is a connected 
	solvable linear algebraic group. However, by the Lie-Kolchin Theorem, 
	see \cite[Theorem~16.30]{Milne:algebraicGroups}, such a group stabilizes a line in any nonzero representation. 
	Since $n\geq 2$, it cannot act irreducibly, providing a contradiction. 
\end{proof}

\subsection{The general case}\label{sec33}

We are now ready to prove Theorem \ref{thm: sigmatranscendence}.  

\begin{proof}[Proof of Theorem \ref{thm: sigmatranscendence}]  
          We argue by induction on $n$. More precisely, our induction assumption reads as follows. 
         \begin{itemize}\item[($\mbox{H}_n$)] 
              For all positive integers $k$ and all $f\in F$ that is a 
              solution to a linear $\phi^{k}$-difference equation of order at most $n$  with coefficients in $K$, 
              we have either $f$ is $\sigma$-transcendental over $K$ or  $f\in K$. 
         \end{itemize}
    Proposition \ref{propo:shift_shift_algrank1} with $b=0$ implies that $(\mbox{H}_{1})$ hold true. 
    Let $n\geq 2$ and let us assume  $(\mbox{H}_{n-1})$.  
    Let $f\in F$ be a solution to a linear $\phi^{k}$-difference equation of order $n$ with coefficients in $K$.  
    Without any loss of generality, we can assume that $f\neq 0$ and $k=1$. 
    Considering the companion matrix $A$ associated with this equation,  
    Proposition \ref{propo:existenceparamfieldsolutionqdiff} ensures the existence of  
    positive integers $r$ and $s$ such that the following properties hold.  

\begin{itemize}
\item[(a)] The vector $(f,\phi(f),\dots, \phi^{n-1} (f))^{\top}$ is a solution to the system $\phi^{r} (Y)=A_{[r]}Y$.  

\item[(b)] There exists a $(\phi^{r},\sigma^s)$-Picard-Vessiot extension $L_A$ for $\phi^r (Y)=A_{[r]}Y$ 
               that is a field and  such that the vector $(f,\phi(f),\dots, \phi^{n-1} (f))^{\top}$  is the first column 
               of a fundamental matrix $U\in\GL_{n}(L_A)$. 
  
\item[(c)]  The Galois group $\G$ of the system $\phi^{r} (Y)=A_{[r]}Y$ over $K$ is connected.
\end{itemize}
By Lemma \ref{lem2},  without loss of generality, up to  replace $\s$ by some  powers, we may reduce to the case where $s=1$. 
    Let us first assume that $\G$ is irreducible. Since $L_A$ is a field, then Proposition \ref{prop: irreducible} 
    shows that $f$ is $\s$-transcendental. Hence $(\mbox{H}_{n})$ holds. From now on, 
    we  assume that  $\G$ is reducible. 
    Furthermore, we assume that $f$ is $\sigma$-algebraic over $K$. Thus, it remains to prove that $f\in K$.  
    Without loss of generality, we can assume that $r=1$. 
 
    By Lemma \ref{lem: reducible1}, there exists a gauge transformation ${T=(t_{i,j})\in \GL_{n}(K)}$ such 
    that 
    $$
    \phi (T)AT^{-1}=\begin{pmatrix}
    A_1& A_{1,2} \\ 
    0 & A_2  
    \end{pmatrix}\, ,
    $$  
    where $A_{i}\in \GL_{n_i}(K)$, $n_1+n_2=n$, and $n_2<n$. 
    Furthermore, let us assume that $n_1$ is minimal with respect to this property. 
    Set  
 \begin{equation}\label{eq:gi}
g= \begin{pmatrix}
g'\\
g''
\end{pmatrix}=
\begin{pmatrix}
 g_1 \\
 \vdots \\
 g_n
 \end{pmatrix}=T \begin{pmatrix}
 f \\
 \vdots \\
 \f^{n-1}(f)
 \end{pmatrix}\in F^n,
  \end{equation}
 where
 $$g'=\begin{pmatrix}
 g_1 \\
 \vdots\\
 g_{n_1}
 \end{pmatrix}\in F^{n_1} \text{ and } g''=\begin{pmatrix}
 g_{n_1+1}\\
 \vdots \\
 g_n
 \end{pmatrix}\in F^{n_2}.
 $$
 Since the coefficients of $T$ belong to $K$,  it follows from Equality \eqref{eq:gi} 
 that  $g\in K^n$ if and only if  $f\in K$. 
From now on, we assume by contradiction that $f\not\in K$. Hence, 
at least one coordinate of $g$ does not belong to $K$.  
 
 We have 
\begin{equation}\label{eq1}
\f\begin{pmatrix}
g' \\
g''
\end{pmatrix}=\begin{pmatrix}
A_1& A_{1,2} \\ 
0 & A_2  
\end{pmatrix}\begin{pmatrix}
g' \\
g''
\end{pmatrix} \,.
\end{equation}
Thus, $g''$ is a solution to the system $\phi(Y)=A_2 Y$. Furthermore, since $f$ is $\sigma$-algebraic over $K$, 
the $g_i$ are also $\sigma$-algebraic over $K$. By ($\mbox{H}_{n_2}$) and Remark~\ref{rem:orderofthesystorderoftheequation}, we obtain that $g''\in K^{n_2}$.

Let $\G_1$ denote the Galois group of the system  $\phi(Y)=A_{1}Y$ over $K$. Let us prove that $\G_1$ is irreducible. 
Indeed, if $\G_1$  were reducible, 
 by Lemma \ref{lem: reducible1}, there would exists a gauge transformation changing $A_1$ into a block upper 
 triangular matrix, contradicting the minimality of $n_1$.   
 Thus the Galois group of $\f(Y)=A_1 Y$ is irreducible.

\begin{claim} 
One has $n_1=1$.
\end{claim}
 
 \begin{proof}[Proof of the claim] 
Let $S_A=K\{U,\frac{1}{\det(U)}\}_\s\subseteq L_A$ denote the $\s$-Picard-Vessiot ring for the system 
$\phi(Y)=AY$. 
Note that, by assumption, $f$ belongs to $S_A$ and thus all the 
$g_{i}$'s  also belong to $S_A$.  
We know that $g''\in K^{n_2}$. Since, by assumption, $g\notin K^{n}$, we have that  $g'\notin K^{n_1}$. 
It follows from the $\s$-Galois correspondence, see Proposition~\ref{propo:spgaloiscorresptransdeg}, 
that there exist a $\C$-$\s$-algebra $B$ and $\tau\in \sGal(L_A/K)(B)$, (i.e. $\tau\colon S_A\otimes_\C B\to S_A\otimes_\C B$ is a $K\otimes_\C B$-$\f\s$-automorphism) with $\tau(g')\neq g'$. 
By \eqref{eq1},
 $\f(g')=A_1g'+A_{1,2}g''$. As $\tau$ commutes with $\f$ and fixes the elements of $K$, 
 we see that $\tau(g')$ 
 is also a solution to $\f(Y)=A_1Y+A_{1,2}g''$. 
 Therefore $g'-\tau(g')\in (S_A\otimes_\C B)^n$ is a nonzero solution of $\f(Y)=A_1Y$. 
Since  $L_A$ is a field, $S_A$ is an integral domain. 
By Corollary \ref{prop1}, $\pi(S_A/K)$, the set of $\s$-algebraic element over $K$ in $S_A$ is invariant under 
the action $\sGal(L_A|K)$ and is a $\s$-ring. 
 Since $g'\in \pi(S_A/K)^{n_1} $, we find successively $\tau(g')\in (\pi(S_A/K)\otimes_\C B)^{n_1}$ and $g'-\tau(g')\in (\pi(S_A/K)\otimes_\C B)^{n_1}$. The latter is a nonzero solution of $\f(Y)=A_1Y$. It follows from Lemma \ref{lemma: fsolutions and base extension} that $\f(Y)=A_1Y$ has a nonzero solution in $\pi(S_A/K)^{n_1}$. As the Galois group of $\f(Y)=A_1Y$ is irreducible and $L_A$ is a field, Proposition \ref{prop: colonne} implies that $n_1=1$.  
\end{proof}

As $n_1=1$, we see that $g'=g_1\in F$ 
is a solution to the inhomogeneous linear order one equation 
$$
\phi (g_1)=A_1 g_1+A_{1,2}g''\,.
$$
Since $g_1\in F$ is $\sigma$-algebraic over $K$ and since $A_1$ and $A_{1,2}g''$ belong to $K$, 
Proposition \ref{propo:shift_shift_algrank1} implies that $g_1\in K$.  
It follows that $g=\begin{pmatrix}
g_1 \\
g''
\end{pmatrix}\in K^n$, and thus 
$f\in K$, providing a contradiction. 
\end{proof}

\appendix
\section{Difference algebraic groups and their actions}\label{appendix}

In this appendix we collect some basic definitions and results concerning difference algebraic groups and their actions that are needed for the proof of Theorem \ref{thm: main}. For a more detailed introduction to difference algebraic groups see \cite[Appendix A]{DVHaWib1}, \cite{Wibmer:FinitenessPropertiesOfAffineDifferenceAlgebraicGroups} and \cite{Wibmer:AlmostsimpleAffineDifferenceAlgebraicGroups}.

\subsection{Difference algebraic groups}\label{sec:sigmagroupscheme}

Difference algebraic groups are the group objects in the category of difference varieties. So we first introduce difference varieties. Throughout this appendix $C$ denotes an arbitrary $\s$-field (not necessarily of characteristic zero). The \emph{$\s$-polynomial ring} $C\{y_1,\ldots,y_n\}$ over $C$ is the polynomial ring over $C$ in the variables $y_1,\ldots,y_n,\s(y_1),\dots,\s(y_n),\ldots$ with action of $\s$ extended from $C$ as suggested by the names of the variables. If $B$ is a $C$-$\s$-algebra and $x=(x_1,\ldots,x_n)\in B^n$, then a $\s$-polynomial $f\in C\{y_1,\ldots,y_n\}$ can be evaluated at $x$ by substituting $\s^i(y_j)$ with $\s^i(x_j)$.
For a subset $F$ of $C\{y_1,\ldots,y_n\}$ we denote the set of solutions of $F$ in $B^n$ with $\V_B(F)$. Note that $B\rightsquigarrow \V_B(F)$ is naturally a functor from the category of $C$-$\s$-algebras to the category of sets.

A difference variety is in essence the set of solutions of a set of $\s$-polynomials. There are two reasons why we cannot simply consider the solutions in some fixed large $\s$-field extension of $C$. Firstly, there is no suitable notion of a $\s$-closure of a $\s$-field, similar to the algebraic closure of a field or the differential closure of a differential field. Secondly, there are many systems of algebraic difference equations where the solutions in $\s$-field extensions reflect very poorly the content of the equations. For example, the system $y+\s(y)=1,\ y\s(y)=0$ has no solution in a $\s$-field. However, $x=(1,0,1,0,\ldots)\in\C^\N$ is a solution in the $\s$-ring of sequences. (Here the action of $\s$ on $\C^\N$ is given by $\s((x_i)_{i\in\N})=(x_{i+1})_{i\in\N}$.)
We therefore define difference varieties as functors, rather than as subsets of some $\widetilde{C}^n$ for some large $\s$-field extension $\widetilde{C}$ of $C$.

\begin{defi}
	A \emph{$\s$-variety} $X$ (over $C$) is a functor from the category of $C$-$\s$-algebras to the category of sets that is isomorphic to a functor of the form $B\rightsquigarrow\V_B(F)$ for some $n\geq 1$ and $F\subset C\{y_1,\ldots,y_n\}$. A \emph{morphism} of $\s$-varieties is a morphism of functors. 
\end{defi}

The functor $X$ given by $B\rightsquigarrow\V_B(F)$ is representable, i.e. there exists a $C$-$\s$-algebra $C\{X\}$ such that $X$ is isomorphic to $\Hom(C\{X\},-)$. Indeed, to specify a solution of $F$ in $B^n$, is equivalent to specifying a morphism of $C$-$\s$-algebras $C\{y_1,\ldots,y_n\}\to B$ that sends all elements in $F$ to zero. The latter is equivalent to specifying a morphism $C\{y_1,\ldots,y_n\}/[F]\to B$, where  $[F]=(\s^i(f)|\ f\in F,\ i\in\N)\subseteq C\{y_1,\ldots,y_n\}$ denotes the $\s$-ideal generated by $F$.
We can therefore choose $C\{X\}=C\{y_1,\ldots,y_n\}/[F]$. In the sequel we will usually identify $X$ and $\Hom(C\{X\},-)$.

A $C$-$\s$-algebra $B$ is \emph{finitely $\s$-generated} if there exists a finite subset $M$ of $B$ such that $B=C\{M\}_\s$. Clearly, $C\{X\}=C\{y_1,\ldots,y_n\}/[F]$ is finitely $\s$-generated. Conversely, every finitely $\s$-generated $C$-$\s$-algebra is isomorphic to one of the form $C\{y_1,\ldots,y_n\}/[F]$. It follows that a functor $X$ from the category of $C$-$\s$-algebras to the category of sets is a $\s$-variety if and only if it is representable by a finitely $\s$-generated $C$-$\s$-algebra. By the Yoneda lemma, this finitely $\s$-generated $C$-$\s$-algebra is uniquely determined by $X$ up to an isomorphism of $C$-$\s$-algebras. We denote it with $C\{X\}$ and call it the \emph{coordinate ring} of $X$. Moreover, a morphism $\eta\colon X\to Y$ of $\s$-varieties corresponds to a morphism $\eta^*\colon C\{Y\}\to C\{X\}$ of $C$-$\s$-algebras. Note that $\eta$ can be recovered from $\eta^*$ via
$\eta_B\colon X(B)=\Hom(C\{X\},B)\to \Hom(C\{Y\},B)=Y(B),\ \psi\mapsto \eta^*\circ\psi$ for any $C$-$\s$-algebra $B$. In summary, we see that the category of $\s$-varieties over $C$ is anti-equivalent to the category of finitely $\s$-generated $C$-$\s$-algebras. In the sequel we will usually identify a $\s$-variety $X$ with $\Hom(C\{X\},-)$.

The category of $\s$-varieties has products. Indeed, if $X$ and $Y$ are $\s$-varieties, then the functor $X\times Y$ given by $B\rightsquigarrow X(B)\times Y(B)$ is represented by $C\{X\}\otimes_C C\{Y\}$. Therefore we can make the following definition.

\begin{defi} \label{defi: salgebraic group}
	A \emph{$\s$-algebraic group} $G$ (over $C$) is a group object in the category of $\s$-varieties (over $C$), i.e. a $\s$-variety $G$ together with morphisms of $\s$-varieties $G\times G\to G$ (multiplication), $G\to G$ (inversion), $1\to G$ (identity) satisfying the group axioms.
\end{defi}

Here $1$ is the functor that associates the trivial group to any $C$-$\s$-algebra $B$. In particular, $G(B)$ is a group for any $C$-$\s$-algebra $B$. A morphism $\eta\colon G\to H$ of $\s$-algebraic groups is a morphism of $\s$-varieties that respect the group structure, i.e. $\eta_B\colon G(B)\to H(B)$ is a morphism of groups for every $C$-$\s$-algebra $B$.

A \emph{$C$-$\s$-Hopf algebra} is a $C$-$\s$-algebra with the structure of a $C$-Hopf algebra such that the Hopf algebra structure maps are morphisms of $C$-$\s$-algebras. Since the category of $\s$-varieties over $C$ is anti-equivalent to the category of finitely $\s$-generated $C$-$\s$-algebras, it follows that the category of $\s$-algebraic groups over $C$ is equivalent to the category of $C$-$\s$-algebras that are finitely $\s$-generated over $C$.

\begin{ex}\label{ex1}
	If $\G$ is an affine algebraic group over $C$, then the functor $[\s]_C\G$ given by $B\rightsquigarrow \G(B^\sharp)$ is a $\s$-algebraic group. Here $B^\sharp$ denotes the $C$-algebra obtained from the $C$-$\s$-algebra by forgetting $\s$. 
	For example, if $V$ is a finite dimensional vector space and $\G=\GL_V$, then $[\s]_C\GL_V$ is the functor that associates to a $C$-$\s$-algebra $B$ the group of $B$-linear automorphism of $V\otimes_C B$. Fixing a basis of $V$, we see that $[\s]_C\GL_V$ is represented by $C\{[\s]_C\GL_{n,C}\}=C\{T_{ij},\frac{1}{\det(T)}\}_\s$, the $\s$-polynomial ring in the variables $T_{ij}$ $(1\leq i,j\leq n)$ localized at the multiplicatively closed subset generated by $\det(T),\ \s(\det(T)),\ldots$. See  \cite[Example 2.6 and Section 1.3]{Wibmer:FinitenessPropertiesOfAffineDifferenceAlgebraicGroups} for a description of $C\{[\s]_C\G\}$ in the general case. In particular, we have an injective morphism $C[\G]\to C\{[\s]_C\G\}$ of $C$-Hopf algebras.
\end{ex}
See \cite[Section 2]{Wibmer:FinitenessPropertiesOfAffineDifferenceAlgebraicGroups} for a list of further examples of $\s$-algebraic groups.

\medskip

If $X$ is a $\s$-variety and $I$ is a $\s$-ideal in $C\{X\}$, we can define a subfunctor $Y$ of $X$ by
$Y(B)=\{\psi\in\Hom(C\{X\},B)|\ \psi(I)=0\}\subset \Hom(C\{X\},B)=X(B)$ for any $C$-$\s$-algebra $B$. Note that $Y$ is a $\s$-variety since it is represented by $C\{X\}/I$. We call $Y$ the \emph{$\s$-closed $\s$-subvariety} of $X$ defined by $I$. For example, if $F\subset F'\subset C\{y_1,\ldots,y_n\}$, then the functor $Y$ given by $B\rightsquigarrow\V_B(F')$ is a $\s$-closed $\s$-subvariety of the functor $X$ given by $B\rightsquigarrow\V_B(F)$, since it corresponds to all morphisms on $C\{X\}=C\{y_1,\ldots,y_n\}/[F]$ that annul the image of $[F']$ in $C\{X\}$. 
For a given $\s$-variety $X$, the $\s$-closed $\s$-subvarieties of $X$ are in bijection with the $\s$-ideals of $C\{X\}$ (\cite[Lemma~1.4]{Wibmer:FinitenessPropertiesOfAffineDifferenceAlgebraicGroups}).

\begin{defi} \label{defi: sclosed subgroup}
	A \emph{$\s$-closed subgroup} $H$ of a $\s$-algebraic group $G$ is a $\s$-closed $\s$-subvariety $H$ of $G$  such that $H(B)$ is a subgroup of $G(B)$ for any $C$-$\s$-algebra $B$. A $\s$-closed subgroup of an affine algebraic group $\G$ is a $\s$-closed subgroup of $[\s]_C\G$.
\end{defi}

\begin{defi} \label{defi: Zariski dense}
	Let $\G$ be an affine algebraic group and let $G$ be a $\s$-closed subgroup of $\G$ (in particular, $G(B)\leq \G(B)$ for all $C$-$\s$-algebras $B$). Then \emph{$G$ is Zariski dense in $\G$} if every regular function on $\G$ that vanishes on $G$ is zero, i.e. for $f\in C[\G]$ with $f(g)=0$ for all $g\in G(B)$ for all $C$-$\s$-algebras $B$, we have $f=0$. 
\end{defi}

In the above definition, we can identify $\G(B)$ with the set of $C$-algebra morphisms from $C[\G]$ to $B$. For $f\in C[\G]$ and $g\colon C[\G]\to B$ a $C$-algebra morphism, we then have $g(f)=f(g)$.

A $\s$-closed subgroup $G$ of $\G$ gives rise to the composition of two morphisms $\psi\colon C[\G]\to C\{[\s]_C\G\}\to C\{G\}$ of $C$-algebras. Note that for a $C$-$\s$-algebra $B$, the inclusion $G(B)=\Hom(C\{G\},B)\to \G(B^\sharp)=\Hom(C[\G],B^\sharp)$ is given by precomposing with $\psi$.
The morphism $\psi$ is injective if and only if $G$ is Zariski dense in $\G$. (This follows by choosing $B=C\{G\}$ and $g=\id\in G(B)=\Hom(C\{G\},C\{G\})$.)

\begin{defn}\label{defn:representationsigmaalggroups}
	Let $G$ be a $\s$-algebraic group over $C$. A \emph{representation} of $G$ is a finite-dimensional $C$-vector space $V$ together with a morphism $G\to [\s]_C\GL_V$ of $\s$-algebraic groups. In particular, for a $C$-$\s$-algebra $B$, every $g\in G(B)$ acts on $V\otimes_C B$ through a $B$-linear automorphism.
	
	A $C$-subspace $W$ of $V$ is \emph{subrepresentation} of $V$ if $g\colon V\otimes_C B\to V\otimes_C B$ maps $W\otimes_C B$ into $W\otimes_C B$ for every $g\in G(B)$ and every $C$-$\s$-algebra $B$.
	
	The representation $V$ of $G$ is \emph{irreducible} if the only subrepresentations of $V$ are the zero subspace and $V$ itself.
\end{defn}

A morphism $f\colon V\to W$ of representations of $G$ is a $C$-linear map such that the diagram 
$$
\xymatrix{
	V\otimes_C B \ar_-g[d] \ar^-{f\otimes B}[r] & W\otimes_C B \ar[d]^g \\
	V\otimes_C B \ar^-{f\otimes B}[r] & W\otimes_C B
}
$$
commutes for every $g\in G(B)$ and every $C$-$\s$-algebra $B$. As for affine algebraic groups (see, for instance, \cite[Section 3.2]{WaterhouseIntrogroupscheme}), the category of representations of $G$ is equivalent to the category of finite dimensional comodules for the Hopf-algebra $C\{G\}$.

\begin{lemma}\label{lemma:irreduciblesigmarepresentationirreduciblezariskiclosurerep}
	Let $V$ be a finite dimensional $C$-vector space and let $\G$ be a closed subgroup of $\GL_V$. If $G$ is a $\s$-closed Zariski dense subgroup of $\G$, then any $G$-subrepresentation of $V$ is a $\G$-subrepresentation of $V$. In particular, if $V$ is irreducible as a representation of $\G$, then $V$ is irreducible as a representation of $G$.
\end{lemma}
\begin{proof}
	Since $G$ is Zariski dense in $\G$ we have $C[\G] \subset C\{G\}$. The comodule structure $\rho\colon V \rightarrow V \otimes_C C\{G\}$ corresponding to the representation of $G$, is obtained from the 
	comodule structure $\overline{\rho}\colon V\to V\otimes_C C[\G]$ corresponding to the representation of $\G$, 
	by composing with the inclusion $V\otimes_C C[\G]\to V\otimes_C C\{G\}$. 
	A $C$-subspace $W$ of $V$ is a $G$-subrepresentation if and only if $\rho(W)\subset W\otimes_C C\{G\}$. Since $\rho$ factors through $V\otimes_C C[\G]$ this implies $\overline{\rho}(W)\subset W\otimes_C C[\G]$, i.e. $W$ is a $\G$-subrepresentation. 
\end{proof}

The following definition introduces the $\s$-dimension of a $\s$-algebraic group.

\begin{defi} \label{defi:sdim}
	Let $B$ be a finitely $\s$-generated $C$-$\s$-algebra and let $M\subseteq B$ be a finite set such that $B=C\{M\}_\s$. For $i\in \N$ let $d_i\in\N$ denote the dimension of the $C$-algebra $C[M,\s(M),\ldots,\s^i(M)]$. Once can show (\cite[Theorem 2.2]{Wibmer:OnTheDimension}) that $\sdim(B)=\lim_{i\to\infty}\frac{d_i}{i+1}$ exists and does not depend on the choice of $M$.
	The \emph{$\s$-dimension} $\sdim(G)$ of a $\s$-algebraic group $G$ is $\sdim(C\{G\})$.
\end{defi}

In general, $\sdim(B)$ need not be an integer. However, one can show (\cite[Theorem 5.1]{Wibmer:OnTheDimension}) that $\sdim(G)$ is always an integer. From this it follows that our definition of the $\s$-dimension of a $\s$-algebraic group agrees with the somewhat more complicated definition given in \cite[Definition~A.25]{DVHaWib1}.

\begin{ex} \label{ex: dimensions agree}
	If $\G$ is a linear algebraic group over $C$, then $\sdim([\s]_C\G)=\dim(\G)$ by \cite[Example 3.10]{Wibmer:FinitenessPropertiesOfAffineDifferenceAlgebraicGroups}.
\end{ex}

\subsection{Difference algebraic elements}\label{subsec:sigmaalgebraic}

The aim of this second part of the appendix is to show that, at least under some mild technical assumptions, the subring $\pi(S|K)$ of a $\s$-Picard-Vessiot ring $S$ over $K$ consisting of all elements of $S$ that are $\s$-algebraic over $K$, is stable under the action of the $\s$-Galois group $G$. It is obvious that every $K$-$\fs$-automorphism of $S$ maps $\pi(S|K)$ into itself, i.e. $\pi(S|K)$ is stable under $G(C)$, where $C=K^\f$. It is much less obvious that $\pi(S|K)$ is stable under the action of $G$ in the strong sense that for every $C$-$\s$-algebra $B$, every element of $G(B)$ maps $\pi(S|K)\otimes_C B$ into $\pi(S|K)\otimes_C B$. 

Let $C$ be $\s$-field and $B$ a $C$-$\s$-algebra. Recall that $a\in B$ is called \emph{$\s$-algebraic over $C$} if $a$ satisfies a nonzero $\s$-polynomial over $C$.  We define
$$\pi(B|C)=\{b\in B |\ b \text{ is $\s$-algebraic over $C$} \}.$$
The following example shows that $\pi(B|C)$ is in general not a subring of $B$.

\begin{ex} \label{ex: pi not closed}
	Let $C=\mathbb{Q}$ (with $\s=\id$) and let $K=\mathbb{Q}(x_1,y_1,x_2,y_2,\ldots)$ be a rational function field in infinitely many variables. Let $B=K^\mathbb{N}$ be the ring of sequences in $K$ that we equip with a structure of $\s$-ring with $\s((a_n)_{n\in\mathbb{N}})=(a_{n+1})_{n\in\mathbb{N}}$. Then $x=(x_1,0,x_2,0\ldots)\in B$ and $y=(0,y_1,0,y_2,\ldots)\in B$ are $\s$-algebraic over $C$ because $x\s(x)=0$ and $y\s(y)=0$. However, $x+y=(x_1,y_1,x_2,y_2,\ldots)\in B$ is in general not $\s$-algebraic over $C$.
\end{ex}

However, it follows from \cite[Theorem 4.1.2 (iii)]{Levin}, that if $B$ is a $\s$-domain, (i.e. $B$ is an integral domain and $\s\colon B\to B$ is injective), then $\pi(B|C)$ is a $C$-$\s$-subalgebra of $B$.
Because of the pathology exhibited in Example \ref{ex: pi not closed}, we modify the definition of $\pi(B|C)$ but in such a way that it remains unchanged for $\s$-domains.

\begin{defi}\label{defi1}
	Let $B$ be a $C$-$\s$-algebra. 
	An element $b\in B$ is \emph{$\s$-bounded} (over $C$) if the sequence $(\dim(C[b,\s(b),\ldots,\s^i(b)]))_{i\in\mathbb{N}}$ is bounded. We set 
	$$\mu(B|C)=\{b\in B|\ \text{$b$ is $\s$-bounded}\}.$$
\end{defi}

The following lemma explains the connection between $\pi(B|C)$ and $\mu(B|C)$.
\begin{lemma}\label{lem1}
Let $B$ be a $C$-$\s$-algebra. Then $\mu(B|C)\subset \pi(B|C)$. Moreover, if $B$ is a $\s$-domain, then $\mu(B|C)=\pi(B|C)$.
\end{lemma}

\begin{proof}
	If $b\in B$ is $\s$-bounded, then $\dim(C[b,\s(b),\ldots,\s^i(b)])<i+1$ for some $i\geq 1$. Thus $b,\s(b),\ldots,\s^i(b)$ are algebraically dependent over $C$ and so $b$ is $\s$-algebraic. Therefore $\mu(B|C)\subset\pi(B|C)$.\par 
	Let us now assume that $B$ is a  $\s$-domain and let us prove the reverse inclusion.
		Let $b\in B$ be $\s$-algebraic over $C$. Because $B$ is a $\s$-domain, its fraction field $L$ is naturally a $\s$-field extension of $C$ and we can consider the $\s$-subfield $C\langle b\rangle_{\s}$ of $L$. According to \cite[Corollary 4.1.18]{Levin}, the transcendence degree $m=\trdeg(C\langle b\rangle_{\s}|C)$ is finite. Thus $\dim(C[b,\s(b),\ldots,\s^i(b)])\leq m$ for all $i\geq 0$. 
\end{proof}
	
We now establish some properties of $\mu(B|C)$ that will permit us later on to deduce properties of $\pi(B|C)$ when $B$ is a $\s$-domain. 

\begin{lemma} \label{lem: tau is subring}
	Let $B$ be a $C$-$\s$-algebra. Then $\mu(B|C)$ is a $C$-$\s$-subalgebra of $B$.
\end{lemma}
\begin{proof}
	Let $a,b\in B$ be $\s$-bounded. Say $\dim(C[a,\ldots,\s^i(a)])\leq m$ and $\dim(C[b,\ldots,\s^i(b)])\leq n$ for all $i\geq 0$. Then $\dim(C[a,b,\ldots,\s^i(a),\s^i(b)])\leq mn$ for all $i\geq 0$. Because $C[a+b,\ldots,\s^i(a+b)]$ and $C[ab,\ldots,\s^i(ab)]$ are contained in the latter $C$-algebra, it follows that $\dim(C[a+b,\ldots,\s^i(a+b)])\leq mn$ and $\dim(C[ab,\ldots,\s^i(ab)])\leq mn$. Thus $\mu(B|C)$ is a subring of $B$. If $b\in B$ is $\s$-bounded, then clearly also $\s(b)$ is $\s$-bounded. Thus $\mu(B|C)$ is a $C$-$\s$-subalgebra of $B$.
\end{proof}

The following lemma is a slight generalization of \cite[Lemma~6.27]{Wibmer:AlmostsimpleAffineDifferenceAlgebraicGroups}. Recall (\cite[Chapter V, \S 17, No. 3, Definition 1]{Bourbaki}) that an algebra $B$ over a field $C$ is called \emph{regular} if $B\otimes_C D$ is an integral domain for every field extension $D$ of $C$.

\begin{lemma} \label{lemma: good fixed points}
	Let $C$ be an inversive $\s$-field (i.e. $\s\colon C\to C$ is surjective) and let $D$ be a $\s$-field extension of $C$ such that $D$ is a regular field extension of $C$. Then there exists a $\s$-field extension $E$ of $D$, such that every element of $E$ that is fixed by all $C$-$\s$-field automorphisms of $E$ lies in $C$, i.e. $E^{\Aut^{\sigma}(E|C)}=C$.
\end{lemma}
\begin{proof}
	We first establish the following claim. There exists a $\s$-field $E$ of $D$ satisfying the following properties. 
	\begin{itemize}
		\item  For every $d\in D\smallsetminus C$ there exists a $C$-$\s$-automorphism $\tau$ of $E$ with $\tau(d)\neq d$.
		\item Every $C$-$\s$-automorphism of $D$ extends to a $C$-$\s$-automorphism of $E$.
		\item The field extension $E|C$ is regular.
	\end{itemize}
	Because $D$ is a regular extension of $C$, the ring $D\otimes_C D$ is an integral domain (\cite[Chapter V, \S 17, No. 3, Proposition 2]{Bourbaki}).  Since $C$ is inversive, it follows from \cite[Corollary 1.6]{TomasicWibmer:Babbitt} that $\s\colon D\otimes_C D\to D\otimes_C D$ is injective. Therefore the field $E$ of fractions of $D\otimes_C D$ is naturally a $\s$-field. We consider $E$ as $\s$-field extension of $D$ via the embedding $d\mapsto d\otimes 1$. The $C$-$\s$-automorphism $\tau$ of $E$ determined by $\tau(d_1\otimes d_2)=d_2\otimes d_1$ satisfies $\tau(d)\neq d$ for every $d\in D\smallsetminus C$. 
	Moreover, every $C$-$\s$-automorphism $\psi$ of $D$ extends to a $C$-$\s$-automorphism of $E$ by $\tau(d_1\otimes d_2)=\tau(d_1)\otimes d_2$. The tensor product of two regular algebras is regular (\cite[Chapter V, \S 17, No. 3, Prop. 3]{Bourbaki}) and the field of fractions of a regular algebra is regular (\cite[Chapter V, \S 17, No. 4, Cor. to Prop. 4]{Bourbaki}). It thus follows that $E$ is a regular extension of $C$ and the claim is established.
	
	Let us now prove the lemma. By the above claim there exists a $\s$-field extension $E_1$ of $D$ such that every $d\in D\smallsetminus C$ can be moved by a $C$-$\s$-field automorphism of $E_1$, every $C$-$\s$-automorphism of $D$ extends to $E_1$ and $E_1|C$ is regular. Applying the claim again (with $E_1|C$ in place of $D|C$) yields a $\s$-field extension $E_2$ of $E_1$ such that every element of $E_1\smallsetminus C$ can be moved by a $C$-$\s$-field automorphism of $E_2$, every $C$-$\s$-automorphism of $E_1$ extends to $E_2$ and $E_2|C$ is regular. Continuing like this, we obtain an ascending chain $C\subset D\subset E_1\subset E_2\subset\cdots$ of $\s$-field extensions. The union $E=\cup_{i=1}^{\infty} E_i$ has the desired property. 
\end{proof}

\begin{lemma} \label{lem: tau and base extension}
	Let $C$ be an inversive $\s$-field and let $B$ be a $C$-$\s$-algebra. Let $D|C$ be an extension of $\s$-fields such that $D$ is a regular field extension of $C$. Then 
	$$\mu(B\otimes_C D|D)=\mu(B|C)\otimes_C D.$$
\end{lemma}
\begin{proof}
	If $b\in B$ is $\s$-bounded over $C$, then $b\otimes 1\in B\otimes_C D$ is $\s$-bounded over $D$. Thus $\mu(B|C)\subseteq\mu(B\otimes_C D|D)$. Since $\mu(B\otimes_C D|D)$ is a $D$-$\s$-subalgebra of $B\otimes_C D$ (Lemma \ref{lem: tau is subring}), it follows that $\mu(B|C)\otimes_C D\subset \mu(B\otimes_C D|D)$. The crucial step to prove the reverse inclusion is to show that $\mu(B\otimes_C D|D)$ descends to $C$.
	
By Lemma \ref{lemma: good fixed points}, there exists a $\s$-field extension $E$ of $D$ such that the fixed field of the group $\Aut^{\s}(E|C)$ of all $\s$-field automorphisms of $E|C$ is $C$. For $\gamma\in \Aut^{\s}(E|C)$ we also denote the induced automorphism $B\otimes_C E\to B\otimes_C E,\ b\otimes\lambda\to b\otimes\gamma(\lambda)$ with $\gamma$. Note that $\gamma\colon B\otimes_C E\to B\otimes_C E$ is an automorphism of $C$-$\s$-algebras that maps $E$ isomorphically onto $E$.
	Assume that $a\in B\otimes_C E$ is $\s$-bounded over $E$. Then $\gamma$ maps $E\{a\}_\s$ isomorphically onto $E\{\gamma(a)\}_\s$ and $E[a,\ldots,\s^i(a)]$ isomorphically onto $E[\gamma(a),\ldots,\s^i(\gamma(a))]$ for every $i\geq 0$. It follows that $\gamma(a)$ is $\s$-bounded over $E$. Consequently $\mu(B\otimes_C E|E)\subseteq B\otimes_C E$ is stable under the action of $\Aut^{\s}(E|C)$.
	
Let $B'$ be the subset of all elements of $\mu(B\otimes_C E|E)$ fixed under the $\Aut^{\s}(E|C)$-action. It is a $C$-$\s$-subalgebra of $B$. 
	 Therefore (cf. \cite[Cor. to Prop. 6, Chapter V, \S 10.4, A.V. 63]{Bourbaki}) $\mu(B\otimes_C E|E)=B'\otimes_C E$. If $b\otimes 1\in B\otimes_C E$ is $\s$-bounded over $E$, then $b$ has to be $\s$-bounded over $C$. Therefore $B'\subset \mu(B|C)$ and it follows that $\mu(B\otimes_C E|E)\subset \mu(B|C)\otimes_C E$ and so $\mu(B\otimes_C E|E)=\mu(B|C)\otimes_C E$. Thus $$\mu(B\otimes_C D|D)\subset \mu(B\otimes_D E|E)=\mu(B|C)\otimes_C E.$$
	But since $\mu(B\otimes_C D|D)\subset B\otimes_C D$, we indeed have the desired inclusion  
	$\mu(B\otimes_C D|D)\subset \mu(B|C)\otimes_C D$.
	\end{proof}

\begin{prop}\label{prop:sigmaalgebraicsigmaintegralGaloisgroup}
	Let $K$ be a $\sphi$-field of characteristic zero such that $\s\colon K\to K$ is surjective and $C=K^{\phi}$ is algebraically closed. Let	
	$S$ be a $\s$-Picard-Vessiot ring over $K$ and let $G$ be the $\s$-Galois group. Assume that $G$ is $\s$-integral, i.e. $C\{G\}$ is an integral domain and $\s\colon C\{G\}\to C\{G\}$ is injective. Then $\mu(S|K)\subset S$ is stable under the $G$-action, i.e. for every $C$-$\s$-algebra $B$, every $\tau\in G(B)$ maps $\mu(S|K)\otimes_C B$ into $\mu(S|K)\otimes_C B$ .
\end{prop}

\begin{proof}
	It suffices to show that the co-action $\rho\colon S\to S\otimes_C C\{G\}$ maps $\mu(S|K)$ into $\mu(S|K)\otimes_C C\{G\}$. Since $\rho(\mu(S|K))\subset\mu(S\otimes_C C\{G\}|K)$, it suffices to show that 
	$\mu(S\otimes_C C\{G\}|K)\subseteq \mu(S|K)\otimes_C C\{G\}$.
	Because $\s\colon K\to K$ is surjective, also $\s\colon C\to C$ is surjective.
	As $C$ is algebraically closed and inversive, also $K\{G\}=K\otimes_C C\{G\}$ is a $\s$-domain (\cite[Cor. A.14 (ii)]{DVHaWib1}). Let $D$ denote the field of fractions of $K\{G\}$. 
	Every integral domain over an algebraically closed field is regular (\cite[Chapter V, \S 17, No. 5, Cor. 2]{Bourbaki}). Thus $C\{G\}$ is regular over $C$. As regularity is preserved under base change (\cite[Chapter V, \S 17, No. 3, Prop. 3]{Bourbaki}), we see that $K\otimes_C C\{G\}$ is regular over $K$. Finally, since the field of fractions of a regular algebra is regular (\cite[Chapter V, \S 17, No. 4, Cor. to Prop. 4]{Bourbaki}), we can conclude that $D$ is a regular field extension of $K$. We can therefore apply Lemma~\ref{lem: tau and base extension} to deduce that $\mu(S\otimes_K D|D)=\mu(S|K)\otimes_KD$. So	
	$$\mu(S\otimes_C C\{G\}|K)=\mu(S\otimes_K K\{G\}|K)\subseteq\mu(S\otimes_K D|D)=\mu(S|K)\otimes_KD.$$
	Therefore
	\begin{multline*}
\mu(S\otimes_K K\{G\}|K)\subseteq (\mu(S|K)\otimes_KD)\cap (S\otimes_K K\{G\})\\
=\mu(S|K)\otimes_K K\{G\}=\mu(S|K)\otimes_C C\{G\}
\end{multline*} as desired.
\end{proof}

\begin{cor}\label{prop1}
Let $K$ be one of the $\sphi$-fields in Cases \textbf{2S}, \textbf{2Q}, and \textbf{2M}. Let $L_A$ be $\s$-Picard-Vessiot extension for $\phi(Y)=AY$ with $A\in \GL_{n}(K)$ and let $S_A\subset L_A$ be the $\sigma$-Picard-Vessiot ring. Assume that $L_A$ is a field. Then $\pi(S_A|K)$ is a $K$-$\f\s$-subalgebra of $S_A$ and stable under the action of the $\s$-Galois group $\sGal(L_A|K)$.
\end{cor}
\begin{proof}
In all cases $\s\colon K\to K$ is surjective and $C=K^\f$ is algebraically closed and of characteristic zero.  By Lemma  \ref{lemma:relatalgclosedbasefieldsigmaintegralPPVgroup}, 
the $\s$-Galois group is $\s$-integral. We can therefore apply Proposition \ref{prop:sigmaalgebraicsigmaintegralGaloisgroup} to deduce that $\mu(S_A|K)$ is stable under the action of the $\s$-Galois group.
By Lemma \ref{lem1}, $\pi(S_A|K)=\mu(S_A|K)$ and $\mu(S_A|K)$ is a $\s$-subring by Lemma \ref{lem: tau is subring}. Because $\f$ and $\s$ commute, $\pi(S_A|K)$ is also stable under $\f$.
\end{proof}

\renewcommand{\labelenumi}{{\rm (\roman{enumi})}}

\setcounter{tocdepth}{2}

\bibliographystyle{alpha}
\bibliography{biblio}

\end{document}